\documentclass{aic}

\usepackage{harpoon}

\theoremstyle{plain}
\newtheorem{theorem}{Theorem}[section]
\newtheorem{conj}[theorem]{Conjecture}
\newtheorem{lemma}[theorem]{Lemma}

\newtheorem{cor}[theorem]{Corollary}
\newtheorem{proposition}[theorem]{Proposition}

\newtheorem{rem}{Remark}[section]
\newtheorem*{slll}{The Lov\'asz local lemma}

\newcounter{claimcount} 
\newtheorem{claim}[claimcount]{Claim}
\newcommand{\resetclaimcount}{\setcounter{claimcount}{0}}

\DeclareMathOperator{\Del}{Del}
\newcommand{\del}{\Del_c}

\DeclareMathOperator{\dom}{dom}

\newcommand{\Prob}{\mathbb{P}}

\newcommand{\sst}[2]{\left\{\,#1 \,\,\middle|\,\, #2\,\right\}}

\DeclareMathOperator{\strong}{s}
\newcommand{\xs}{\chi'_{\strong}}
\newcommand{\E}{\mathbb{E}}
\newcommand{\Eon}[1]{\underset{#1}{\E}}
\newcommand{\N}{\mathbb{N}}
\newcommand{\R}{\mathbb{R}}
\newcommand{\luv}{\ell_{uv}}
\newcommand{\lvw}{\ell_{vw}}
\newcommand{\luw}{\ell_{uw}}

\newcommand{\I}{{\mathcal{I}}}
\newcommand{\bI}{{\bf{I}}}
\newcommand{\A}{{\bf{A}}}
\newcommand{\Pk}{\Prob_{\text{keep}}}
\newcommand{\B}{{\mathcal{B}}}

\newcommand{\codeg}{\Lambda}
\newcommand{\cosigma}{\widehat{\sigma}}

\newcommand{\clsigma}{\dot{\sigma}}

\newcommand{\eps}{\varepsilon}

\newcommand{\GBJ}{\overrightharp{G}_1}
\newcommand{\GHJK}{\overrightharp{G}_2}

\aicAUTHORdetails{%
  title = {An Improved Procedure for Colouring Graphs of Bounded Local Density}, 
  author = {Eoin Hurley, R\'emi de Joannis de Verclos, and Ross J. Kang},
  plaintextauthor = {Eoin Hurley, Remi de Joannis de Verclos, and Ross J. Kang},
  plaintexttitle = {An Improved Procedure for Colouring Graphs of Bounded Local Density}, 
  runningtitle = {Colouring Graphs of Bounded Local Density}, 
  runningauthor = {Eoin Hurley, R\'emi de Joannis de Verclos, and Ross J. Kang},
  copyrightauthor = {E. Hurley, R. de Joannis de Verclos, and R. J. Kang},
  keywords = {graph colouring, local sparsity, Erd\H{o}s--Ne\v{s}et\v{r}il conjecture, Reed's conjecture},
}   

\aicEDITORdetails{%
   year={2022},
   number={7},
   received={13 October 2020},   
   published={22 September 2022},  
   doi={10.19086/aic.2022.7},      
}   

\begin{document}

\begin{frontmatter}[classification=text]

\title{An Improved Procedure for Colouring Graphs of Bounded Local Density\titlefootnote{A preliminary version of this paper appeared in Proceedings of the 2021 ACM-SIAM Symposium on Discrete Algorithms (SODA 2021, Alexandria, USA): 135-148, 2021.}} 

\author[eoin]{Eoin Hurley\thanks{Part of this research was carried out while the author was a master's student at Utrecht University.}}
\author[remi]{R\'emi de Joannis de Verclos\thanks{Supported by a Vidi grant (639.032.614) of the Netherlands Organisation for Scientific Research (NWO). This research was conducted while these authors were at Radboud University.}}
\author[ross]{Ross J. Kang\footnotemark[2]}

\begin{abstract}
We develop an improved bound for the chromatic number of graphs of maximum degree $\Delta$ under the assumption that the number of edges spanning any neighbourhood is at most $(1-\sigma)\binom{\Delta}{2}$ for some fixed $0<\sigma<1$.
The leading term in the reduction of colours achieved through this bound is best possible as $\sigma\to0$.
As two consequences, we advance the state of the art in two longstanding and well-studied graph colouring conjectures, the Erd\H{o}s--Ne\v{s}et\v{r}il conjecture and Reed's conjecture.
We prove that the strong chromatic index is at most $1.772\Delta^2$ for any graph $G$ with sufficiently large maximum degree $\Delta$.
We prove that the chromatic number is at most $\lceil 0.881(\Delta+1)+0.119\omega\rceil$ for any graph $G$ with clique number $\omega$ and sufficiently large maximum degree $\Delta$.
Additionally, we show how our methods can be adapted under the additional assumption that the codegree is at most $(1-\sigma)\Delta$, and establish what may be considered first progress towards a conjecture of Vu.
\end{abstract}
\end{frontmatter}

\section{Introduction}

This paper follows in a long line of investigation of the following Ramsey-type graph colouring problem.
\begin{quote}\em
What is the best upper bound on the chromatic number $\chi$ for graphs of given maximum degree $\Delta$ and given maximum local density --- that is, with neighbourhood subgraphs each inducing at most a certain edge density?
\end{quote}
This deep and elegant problem has its roots going back more than half a century~\cite{Viz68}.
The archetypal result of this type is one of Johansson~\cite{Joh96} that was recently sharpened with the entropy compression method by Molloy~\cite{Mol19} as follows: any graph $G$ that is triangle-free --- that is, with a maximum local density of precisely zero --- has chromatic number satisfying $\chi(G)\le(1+o(1))\Delta(G)/\log\Delta(G)$ as $\Delta(G)\to\infty$.
This betters by a logarithmic factor the trivial upper bound $\chi(G)\le \Delta(G)+1$ that holds for {\em any} graph $G$.
It is sharp up to a (small) constant multiple due to random regular graphs.
Moreover, it yields as a direct corollary the asymptotically best to date upper bounds on the off-diagonal Ramsey numbers due to Shearer~\cite{She83}.
Alon, Krivelevich and Sudakov~\cite{AKS99} bootstrapped Johansson's theorem to show a more general bound under the condition of maximum local density at most $1/f=o(1)$. This too has been refined recently~\cite{DKPS20+} using elementary properties of the hard-core model (cf.~also~\cite{DJKP21} and~\cite{DKPS20+b}) as follows: any graph $G$ with local density at most $1/f$, where $f=f(\Delta(G))$, $f \to\infty$ as $\Delta(G)\to\infty$, and $f\le \binom{\Delta(G)}2+1$, has chromatic number satisfying $\chi(G)\le(1+o(1))\Delta(G)/\log\sqrt{f}$. Note that this statement includes the triangle-free one as a special case with $f=\binom{\Delta(G)}2+1$. While in that result the condition $1/f=o(1)$ excludes the possibility of a neighbourhood subgraph having 
nonnegligible density, here it is this `denser' situation which will be our primary focus.

To specify how far we are from a trivial local density condition, we adopt the following notation. Given $\sigma>0$, a graph $G$ is said to be {\em $\sigma$-sparse} if for every $v\in V(G)$ the subgraph $G[N(v)]$ induced by the neighbourhood $N(v)$ of $v$ has at most $(1-\sigma)\binom{\Delta(G)}2$ edges.
As a means towards progress in a problem of Erd\H{o}s and Ne\v{s}et\v{r}il (which we discuss in further detail later on in the paper), Molloy and Reed~\cite{MoRe97} initiated the study of the chromatic number of $\sigma$-sparse graphs. In particular, using a ``na\"ive'' probabilistic colouring procedure, they showed the following.

\begin{theorem}[Molloy and Reed~\cite{MoRe97}]\label{thm:MoResparse}
There is a positive function $\eps_{\ref{thm:MoResparse}}=\eps_{\ref{thm:MoResparse}}(\sigma)$ such that the following holds.
For each $0<\sigma\le1$ there is $\Delta_0$ such that the chromatic number satisfies $\chi(G)\le (1-\eps_{\ref{thm:MoResparse}})\Delta(G)$ for any $\sigma$-sparse graph $G$ with $\Delta(G)\ge\Delta_0$.
\end{theorem}
\noindent
Note this constitutes a constant factor improvement upon the trivial upper bound in this case.

Our work marks important progress in the quantitative optimisation of Theorem~\ref{thm:MoResparse} for $\sigma$-sparse graphs, i.e.~in pursuit of the maximum $\eps_{\ref{thm:MoResparse}}$ as a function of $\sigma$. We briskly survey the landscape prior to our work. Molloy and Reed themselves proved Theorem~\ref{thm:MoResparse} for 
$\eps_{\ref{thm:MoResparse}} \ge 0.0238\sigma$.
It was over two decades before Bruhn and Joos~\cite{BrJo18} were able to improve upon this by establishing that 
$\eps_{\ref{thm:MoResparse}} \ge 0.1827\sigma-0.0778\sigma^{3/2}$.
Soon after, Bonamy, Perrett and Postle~\cite{BPP22} improved this further through an iterative approach (that also captured the more general notions of list and correspondence colouring), and showed that  
$\eps_{\ref{thm:MoResparse}} \ge 0.3012\sigma-0.1283\sigma^{3/2}$.
Our contribution is the analysis of a different iterated colouring procedure, one that gradually releases usable colours, and is based around a random priority assignment strategy. This shows Theorem~\ref{thm:MoResparse} is true for any 
$\eps_{\ref{thm:MoResparse}} < \sigma/2 - \sigma^{3/2}/6$.
\begin{theorem}
\label{col_result}
Define $\eps_{\ref{col_result}} = \eps_{\ref{col_result}}(\sigma) = \sigma/2 - \sigma^{3/2}/6$.
For each $\iota>0$ and $0<\sigma\le1$, there is $\Delta_{\ref{col_result}}=\Delta_{\ref{col_result}}(\iota)$ such that the chromatic number satisfies $\chi(G)\le (1-\eps_{\ref{col_result}}(\sigma)+\iota)\Delta(G)$ for any $\sigma$-sparse graph $G$ with $\Delta(G)\ge\Delta_{\ref{col_result}}$.
\end{theorem}

\noindent
In the densest cases, as $\sigma\to0$, the leading coefficient $1/2$ in the expression for $\eps_{\ref{col_result}}$ is best possible, as certified by the following simple construction, cf.~also~\cite[Ex.~10.1]{MoRe02}.

\begin{proposition}\label{prop:sharp}
For each $\Delta\ge 1$ and  $\sigma>0$, there is a $\sigma$-sparse graph $G^\Delta_\sigma$ of maximum degree $\Delta$ such that as $\sigma\to0$ its chromatic number satisfies
\(
1-\chi(G^\Delta_\sigma)/\Delta = \sigma/2+O(\sigma^2).
\)
\end{proposition}
\begin{proof}
Let $G^\Delta_\sigma$ consist of a clique of size $\min\{1,\lfloor\sqrt{1-\sigma}\cdot\Delta\rfloor\}$ with $\Delta+1-\min\{1,\lfloor\sqrt{1-\sigma}\cdot\Delta\rfloor\}$ vertices of degree one appended to each vertex in the clique. It is trivial to verify that this graph has maximum degree $\Delta$, is $\sigma$-sparse, and has chromatic number $\min\{1,\lfloor\sqrt{1-\sigma}\cdot\Delta\rfloor\}$. The conclusion follows from a Taylor expansion of $1-\sqrt{1-\sigma}$ at $\sigma=0$.
\end{proof}

\noindent
Obtaining sharpness of this accuracy is considered uncommon for such Ramsey-type problems.
On the other hand, in the sparser cases as $\sigma\to1$ one might hope for a guarantee on $\eps_{\ref{thm:MoResparse}}(\sigma)$ that approaches $1$, since Johansson's result for triangle-free graphs implies for $\sigma=1$ that Theorem~\ref{thm:MoResparse} holds for any $\eps_{\ref{thm:MoResparse}}(1) < 1$. Thus there is still room for improvement, since the methods we have employed here only imply for that special case that Theorem~\ref{thm:MoResparse} holds for any $\eps_{\ref{thm:MoResparse}}(1) < 1/3$.

Nevertheless the improved bound of Theorem~\ref{col_result} yields state-of-the-art bounds in two well-known and longstanding conjectures, namely the Erd\H{o}s--Ne\v{s}et\v{r}il conjecture and Reed's conjecture.
In Subsections~\ref{sub:ErdosNesetril} and~\ref{sub:Reed}, we discuss these consequences.

Moreover, a light adaptation of our methods can handle another basic, but stronger notion of local sparsity, namely, that the neighbourhoods induce subgraphs of at most a certain maximum degree. This yields a bound of similar sharpness to Theorem~\ref{col_result}, which may be considered as first positive progress towards a conjecture of Vu~\cite{Vu02}. We state this result and discuss its context in Subsection~\ref{sub:Vu}.

The proof of Theorem~\ref{col_result} 
is provided in full in Section~\ref{sec:proof}.
For the convenience of the reader, we have prefaced Section~\ref{sec:proof} with a succinct overview of the ideas and methods.

\subsubsection*{Structure of the paper}

Subsections~\ref{sub:ErdosNesetril},~\ref{sub:Reed} and~\ref{sub:Vu} describe the background to the Erd\H{o}s--Ne\v{s}et\v{r}il, Reed's and Vu's conjectures, respectively, and the progress we obtain either directly through Theorem~\ref{col_result} or through its adaptation.
Subsection~\ref{sub:prelim} lists some notation and tools we use.
We give an outline of the proof of Theorem~\ref{col_result} at the beginning of Section~\ref{sec:proof}.
The remainder of Section~\ref{sec:proof} provides the full proof.
In Section~\ref{sec:applications}, we give further details of the two direct applications of Theorem~\ref{col_result} as well as the adaptation of Theorem~\ref{col_result} towards Vu's conjecture.
In Appendix~\ref{sec:talagrand}, we include some auxiliary argumentation necessary for our concentration inequality applications.

\subsection{A step towards the Erd\H{o}s--Ne\v{s}et\v{r}il conjecture}\label{sub:ErdosNesetril}

Given a graph $G$, an {\em induced matching} is a subset $M$ of the edges of $G$ such that for any pair $e,e'$ of distinct edges of $M$, neither $e$ and $e'$ are incident nor are any of the four possible edges between an endpoint of $e$ and an endpoint of $e'$ present in $G$.
A {\em strong edge-colouring} of $G$ is a partition of its edge set $E(G)$ into induced matchings of $G$.
The strong chromatic index $\xs(G)$ of $G$ is the least number of parts needed in any strong edge-colouring of $G$. Equivalently, $\xs(G)$ is the chromatic number $\chi(L(G)^2)$ of the square $L(G)^2$ of the line graph $L(G)$ of $G$. (The square of a graph is obtained from the graph itself by adding edges between all pairs of distinct nonadjacent vertices that are connected by a two-edge path.)
In the 1980s (cf.~\cite{Erd88}), Erd\H{o}s and Ne\v{s}et\v{r}il proposed the problem of bounding $\xs(G)$ in terms of the maximum degree $\Delta(G)$ of $G$. Since the maximum degree $\Delta(L(G)^2)$ of the square of the line graph of $G$ is at most $2\Delta(G)(\Delta(G)-1)$, the strong chromatic index is trivially bounded by $\xs(G)\le 2\Delta(G)^2-2\Delta(G)+1$. They conjectured something much stronger.

\begin{conj}[Erd\H{o}s and Ne\v{s}et\v{r}il, cf.~\cite{Erd88}]\label{conj:ErNe}
The strong chromatic index satisfies $\xs(G)\le 1.25\Delta(G)^2$ for all $G$.
\end{conj}

\noindent
(See~\cite{JKP19,CaKa19} for a fascinating strengthened, yet essentially equivalent, form of this conjecture.)
If true, this bound would be exact for a suitable blow-up of the 5-edge cycle (in the $\Delta(G)$ even case).
It was more than a decade before a breakthrough by Molloy and Reed~\cite{MoRe97} yielded some absolute constant $\eps>0$ such that $\xs(G)\le (2-\eps)\Delta(G)$ for all $G$.
More specifically, they proved the following statement.

\begin{theorem}[Molloy and Reed~\cite{MoRe97}]\label{thm:MoRestrong}
There is some $\eps_{\ref{thm:MoRestrong}}>0$ and some $\Delta_0$ such that the strong chromatic index satisfies $\xs(G)\leq (2-\eps_{\ref{thm:MoRestrong}})\Delta(G)^2$ for any graph $G$ with $\Delta(G)\ge \Delta_0$.
\end{theorem}

\noindent
We may bound the absolute constant $\eps>0$ mentioned just above by comparing the bound of Theorem~\ref{thm:MoRestrong} with the trivial bound on $\xs(G)$ when $\Delta(G)<\Delta_0$.
Molloy and Reed proved that $\eps_{\ref{thm:MoRestrong}} \ge 0.001$.
A key insight they made in their proof of Theorem~\ref{thm:MoRestrong} was to split the task into two separate subtasks, first, showing for some absolute constant $\sigma>0$ that $L(G)^2$ is $\sigma$-sparse for any graph $G$, and, second, showing a nontrivial improvement on the trivial colouring bound under the assumption of $\sigma$-sparsity, i.e.~Theorem~\ref{thm:MoResparse}.
Bruhn and Joos~\cite{BrJo18} were the first to revisit this problem, and they not only significantly improved on the estimate of $\eps_{\ref{thm:MoResparse}}$ in Theorem~\ref{thm:MoResparse} as mentioned earlier, but also proved an asymptotically extremal lower bound on $\sigma>0$ such that $L(G)^2$ is $\sigma$-sparse for all $G$. In this way, they obtained that $\eps_{\ref{thm:MoRestrong}} \ge 0.070$.
The more recent work of Bonamy {\em et al.}~\cite{BPP22} obtained further improvements. As mentioned earlier, they improved the estimate of $\eps_{\ref{thm:MoResparse}}$ in Theorem~\ref{thm:MoResparse} through an iterative approach. They were moreover able to improve on the separation into two subtasks, by showing better sparsity on a subgraph of $L(G)^2$ according to a degeneracy-type argument. Through this, they obtained that $\eps_{\ref{thm:MoRestrong}} \ge 0.165$.
By combining Theorem~\ref{col_result} with this last-mentioned method, we derive that $\eps_{\ref{thm:MoRestrong}} \ge 0.228$. The proof is given in Subsection~\ref{sub:ErdosNesetrilproof}.

\begin{theorem}
\label{strong_bound}
There is some $\Delta_0$ such that the strong chromatic index satisfies $\xs(G)\leq 1.772\Delta(G)^2$ for any graph $G$ with $\Delta(G)\ge \Delta_0$.
\end{theorem}

\noindent
We humbly agree that the above sequence of improvements on estimates for $\eps_{\ref{thm:MoRestrong}}$ suggests that the hypothetically optimal determination $\eps_{\ref{thm:MoRestrong}}=0.75$ remains far from reach. Even a proof of $\eps_{\ref{thm:MoRestrong}}$ being $0.75$ might leave open the nontrivial task of proving Conjecture~\ref{conj:ErNe} for all graphs with maximum degree less than $\Delta_0$. Despite sustained and considerable efforts, so far it has only been established for graphs of maximum degree at most $3$~\cite{And92,HQT93}.

\subsection{A step towards Reed's conjecture}\label{sub:Reed}

Another Ramsey-type problem (perhaps even closer to quantitative Ramsey theory) asks the following.
\begin{quote}\em
What is the best upper bound on the chromatic number $\chi$ for graphs of given maximum  degree $\Delta$ and given clique number $\omega$?
\end{quote}

\noindent
An aforementioned result of Johansson~\cite{Joh96} has settled this question up to a constant multiple as $\Delta\to\infty$ when $\omega=2$. Already the case $\omega=3$ is open and difficult. This is closely related to an important conjecture of Ajtai, Erd\H{o}s, Koml\'os and Szemer\'edi~\cite{AEKS81}. For $\omega$ asymptotically smaller than $\Delta$ (as $\Delta\to\infty$), the current best bounds were recently obtained in~\cite{DKPS20+}, improving upon another important result of Johansson~\cite{Joh96b} (as well as recent improvements, e.g.~in~\cite{Mol19}) by a constant factor. Again, here we will be mostly concerned with a `denser' regime, namely, when $\omega$ is linear in $\Delta$. Related to this, Reed proposed an evocative conjecture.

\begin{conj}[Reed~\cite{Ree98}]\label{conj:Ree98}
The chromatic number satisfies $\chi(G) \le \lceil \frac12(\omega(G)+\Delta(G)+1)\rceil$ for any graph $G$.
\end{conj}

\noindent
In other words, he asked if the chromatic number $\chi(G)$ of a graph $G$ is always at most the average, rounded up, of the trivial lower bound, $\omega(G)$, and the trivial upper bound, $\Delta(G)+1$, for $\chi(G)$.
If true, the bound is sharp, for instance, for the Chv\'atal graph. The bound is trivially true when $\omega(G)\ge\Delta(G)$ and it follows from Brooks' theorem~\cite{Bro41} for $\omega(G)=\Delta(G)-1$. 
In~\cite{DKPS20+}, it was shown that, if $\omega(G) \le \Delta(G)^c$ for some fixed $c<1/100$, then the bound holds provided $\Delta(G)$ is sufficiently large. (There is some room in the method there to increase the constant $1/100$ slightly, but not above $1/16$ without additional ideas.)
Curiously, despite Johansson's result, the conjecture is still open in the special case $\omega(G)=2$, particularly for small values of $\Delta(G)$.
As evidence towards his conjecture, Reed succeeded in proving, through a lengthy set of arguments that are probabilistic in nature, the following.

\begin{theorem}[Reed~\cite{Ree98}]\label{thm:Ree98}
There is some $\eps_{\ref{thm:Ree98}}>10^{-8}$ and some $\Delta_{\ref{thm:Ree98}}$ such that the chromatic number satisfies $\chi(G) \le \frac12(\omega(G)+\Delta(G)+1)$ for any graph $G$ satisfying $\omega(G) \ge (1-\eps_{\ref{thm:Ree98}})\Delta(G)$ and $\Delta(G)\ge \Delta_{\ref{thm:Ree98}}$.
\end{theorem}

\noindent
Note that this statement implies that some (barely) nontrivial convex combination of $\omega(G)$ and $\Delta(G)+1$ suffices as an upper bound for $\chi(G)$.

\begin{cor}\label{cor:Ree98}
There is some $\eps_{\ref{cor:Ree98}}\ge \eps_{\ref{thm:Ree98}}/2$ and some $\Delta_0$ such that the chromatic number satisfies  $\chi(G) \le \lceil (1-\eps_{\ref{cor:Ree98}})(\Delta(G)+1)+\eps_{\ref{cor:Ree98}}\omega(G)\rceil$ for any graph $G$ with $\Delta(G)\ge \Delta_0$.
\end{cor}

\noindent
Reed himself made little effort to optimise the value of $\eps_{\ref{cor:Ree98}}$, but noted that it cannot be more than $1/2$ by a standard probabilistic construction.
Bonamy {\em et al.}~\cite{BPP22} recently revisited this problem and showed $\eps_{\ref{cor:Ree98}} > 0.038$. Delcourt and Postle~\cite{DePo17} have announced that $\eps_{\ref{cor:Ree98}} > 0.076$. One consequence of Theorem~\ref{col_result}, combined with a claimed result of~\cite{DePo17}, is an improvement on these estimates, in particular, that $\eps_{\ref{cor:Ree98}} \ge 0.119$. The proof is given in Subsection~\ref{sub:Reedproof}.

\begin{theorem}
\label{Reed_bound}
There is some $\Delta_0$ such that the chromatic number satisfies $\chi(G) \le \lceil 0.881(\Delta(G)+1)+0.119\omega(G)\rceil$ for any graph $G$ with $\Delta(G)\ge \Delta_0$.
\end{theorem}

\subsection{A step towards Vu's conjecture}\label{sub:Vu}

Recall that the {\em codegree} of two vertices $u,v$ is the number of distinct neighbours in common to both $u$ and $v$.
One further Ramsey-type problem asks the following.
\begin{quote}\em
What is the best upper bound on the chromatic number $\chi$ for graphs of given maximum degree $\Delta$ and given maximum codegree $\codeg$?
\end{quote}
If we moreover restrict our attention to codegree taken among pairs of endpoints, then this problem yet again concerns the colouring of graphs of bounded local density, for $\codeg$ is an upper bound on the maximum degree in any neighbourhood subgraph.
Thus having a bounded codegree condition is stronger than having bounded local edge density.
Our interest is in $\codeg$ being some nontrivial proportion of $\Delta$, say, $\codeg=(1-\cosigma)\Delta$ for some fixed $0<\cosigma<1$.
For the case $\cosigma=1$ Johansson's theorem~\cite{Joh96} resolves this question up to a constant multiple as $\Delta\to\infty$, while the trivial upper bound corresponds to the case $\cosigma=0$.
Because of its relationship to an important result of Kahn on the list chromatic index of linear hypergraphs~\cite{Kah96} and in turn to the Erd\H{o}s--Faber--Lov\'asz conjecture (cf.~\cite{Erd81}), Vu~\cite{Vu02} essentially proposed the following~\footnote{In fact, he posed it in a stronger form in terms of list colouring, but noted the analogous statement for independence number (which is weaker than the statement of Conjecture~\ref{conj:Vu}) is also open, as remains the case to this day.}.

\begin{conj}[Vu~\cite{Vu02}]\label{conj:Vu}
For each $\iota>0$ and $0\le\cosigma\le1$ there is $\Delta_0$ such that the chromatic number satisfies $\chi(G) \le (1-\cosigma+\iota)\Delta$ for any graph $G$ with maximum codegree at most $(1-\cosigma)\Delta(G)$ and $\Delta(G) \ge \Delta_0$.
\end{conj}

\noindent
This bound if true would be asymptotically best possible as certified, for instance, by a clique of size $\lfloor (1-\cosigma)\Delta\rfloor+1$.
Vu was mainly interested in Conjecture~\ref{conj:Vu} in the sparse regime with $\cosigma$ close to $1$. On the other hand, here we make marked progress in the dense regime with $\cosigma$ close to $0$. By adapting the proof method for Theorem~\ref{col_result}, we show the following.

\begin{theorem}\label{Vu_bound}
Define $\eps_{\ref{Vu_bound}} = \eps_{\ref{Vu_bound}}(\cosigma) = \max\{ \cosigma/(1+2\cosigma) - (2\cosigma)^{3/2},\eps_{\ref{col_result}}(\cosigma)\}$.
For each $\iota>0$ and $0\le\cosigma\le1$, there is $\Delta_{\ref{Vu_bound}}=\Delta_{\ref{Vu_bound}}(\iota)$ such that the chromatic number satisfies $\chi(G)\le (1-\eps_{\ref{Vu_bound}}(\cosigma)+\iota)\Delta(G)$ for any graph $G$ with  maximum codegree at most $(1-\cosigma)\Delta(G)$ and $\Delta(G)\ge\Delta_{\ref{Vu_bound}}$.
\end{theorem}

\noindent
We remark that the first of the two terms in the maximisation defining $\eps_{\ref{Vu_bound}}$ is the greater one as long as $\cosigma$ is at most around $0.028$.
As $\cosigma\to 0$, note that $\cosigma/(1+2\cosigma) - (2\cosigma)^{3/2} = \cosigma+o(\cosigma)$, and so the leading coefficient in the expression for $\eps_{\ref{Vu_bound}}$ is $1$, which we noted after Conjecture~\ref{conj:Vu} is best possible. As such Theorem~\ref{Vu_bound} constitutes, to the best of our knowledge, the first direct advance towards Conjecture~\ref{conj:Vu}.

\subsection{Graph theoretic notation and probabilistic preliminaries}
\label{sub:prelim}
\label{sub:notation}

Throughout the paper we have adopted the following notation. 

For $k\in \N$, let $[k]$ denote the set $\{1,2,\dots,k\}$.

Given a graph $G$ and a vertex $v$, we write $N_G(v)$ for the {\em (open) neighbourhood} $\{u\in V(G):uv\in E(G)\}$ of $v$ and $N_G[v]$ for the {\em closed neighbourhood} $N_G(v)\cup \{v\}$ of $v$ in $G$. 
The {\em degree} of $v$ in $G$ is denoted by $d_G(v)=|N_G(v)|$.
We usually drop the subscript when there is no ambiguity.

Given a graph $G$ and a vertex subset $S\subseteq V(G)$, we write $G[S]$ for the subgraph of $G$ induced by $S$.

Given a graph $G$, a {\em list-assignment} for $G$ is a map $L:V(G)\to 2^{\N}$, where $2^{\N}$ by convention denotes the set of all subsets of $\N$. We call a list-assignment $L$ a {\em $k$-list-assignment} if $|L(v)|=k$ for all $v\in V(G)$, i.e.~a $k$-list-assignment is a map $L:V(G)\to \binom{\N}{k}$, where $\binom{\N}{k}$ by convention denotes the set of all subsets of $\N$ of size $k$.
We call $L(v)$ the {\em list} of the vertex $v$. 
Given a list-assignment $L$ of $G$, a {\em partial proper $L$-colouring} of $G$ is a map $c:U\to \N$, where $U\subseteq V(G)$, such that $c(v)\in L(v)$ for all $v\in V(G)$ and $c(v)\neq c(w)$ for any $vw\in E(G)$.
We write $\dom(c)$ for the domain of $c$ and drop `partial' if $\dom(c)=V(G)$.
Note that the existence of a proper $L$-colouring for any constant $k$-list-assignment $L$ of $G$ is equivalent to the assertion $\chi(G) \le k$.

Since we will be interested in gradually building up partial proper $L$-colourings, we introduce some terminology to describe the process.
Given a list-assignment $L$ and a partial proper $L$-colouring $c$ of $G$, the {\em residual subgraph} $G_c$ of $G$ with respect to $c$ is the induced subgraph $G[V(G)\setminus \dom(c)]$ and the {\em residual list-assignment} $L_c:G_c\to 2^{\N}$ is defined by $L_c(v) = L(v)\setminus c(N(v))$ for all $v\in V(G)$.
Note that if $c'$ is a proper $L_c$-colouring of $G_c$, then the union of the colourings $c$ and $c'$ is a proper $L$-colouring of $G$.

Our proofs rely on probabilistic methods, for which we require certain probabilistic tools. We use the following form of the Lov\'asz local lemma~\cite{ErLo75}.

\begin{slll}[\cite{ErLo75}]
    Let $p\in[0,1)$, and $\mathcal{A}$ be a finite set of ``bad'' events so that for every $A\in\mathcal{A}$ 
    \begin{itemize}
    \item $\Prob[A]\leq p$, and 
    \item $A$ is mutually independent of all but at most $d$ other events in $\mathcal{A}$.
    \end{itemize}
    If $4pd\leq 1$, then the probability that none of the (``bad'') events in $\mathcal{A}$ occur is strictly positive. 
\end{slll}

To help bound the probability of ``bad'' events in our application of the local lemma, we need to prove concentration of measure.
If $\Omega$ is a product of discrete spaces, we can define smoothness as the property that if $\omega\in\Omega$ and $\omega'\in\Omega$ differ in only one coordinate then $|X(\omega)-X(\omega')|<c$. Talagrand's inequality~\cite{Tal95} tells us that such smooth random variables are highly concentrated. 
However, some random variables that arise from our colouring procedure are not smooth and it is possible for one vertex to cause many others to be uncoloured. Fortunately, such a situation is highly unlikely, one might say exceptional, and can be handled by an adaption of Talagrand's inequality due to Bruhn and Joos~\cite{BrJo18}.

We can formalise this notion as follows, let $\Omega$ be a product space of probability spaces and let $\Omega^*\subseteq \Omega$ be a set of \emph{exceptional} outcomes. We say that $X$ has upward $(s,c)$-certificates if for every for  $\omega \in \Omega\backslash\Omega^*$ we have an index set $I$ of size at most $s$ that identifies all the influences on $X(\omega)$ such that for another event $\omega' \in \Omega\backslash\Omega^*$ if $\omega|_{I}$ differs from $\omega'|_{I}$ in fewer than ${t}/{c}$ coordinates, then $X(\omega')> X(\omega)-t$. 
In other words, for an \emph{unexceptional} $\omega$ any increase in $X(\omega)$ comes from changes in a not too large set of coordinates indexed by $I$, and none of these coordinates increase $X(\omega)$ too much. 
\begin{theorem}[Bruhn and Joos~\cite{BrJo18}, cf.~Talagrand~\cite{Tal95}]\label{thm:upward certs}
Let $((\Omega_i,\sigma_i,\Prob_i))_{i=1}^n$ be probability spaces, $(\Omega,\sigma,\Prob)$ be their product space and $\Omega^*\subseteq\Omega$ a set of \emph{exceptional} outcomes. Let $X:\Omega\to \R$ be a random variable, $M=\max\{\sup|X|,1\}$, and $c\geq 1$. If  $ \Prob[\Omega^*]\leq M^{-2}$ and $X$ has upward $(s,c)$-certificates then for $t>50c\sqrt{s}$,
\begin{equation}
    \Prob[|X-\E[X]|\geq t]\leq4e^{-\frac{t^2}{16c^2s}}+4\Prob[\Omega^*].
\end{equation}
\end{theorem}

\begin{rem}
One might argue that Bruhn and Joos proved Theorem~\ref{thm:upward certs} under the additional assumption that the product space is a product of {\em discrete} probability spaces. This is because, for the avoidance of distracting measurability issues, that assumption was also taken by Talagrand in~\cite{Tal95}. Since our applications of Theorem~\ref{thm:upward certs} are prevented by such an assumption, we provide in Appendix~\ref{sec:talagrand} a derivation of the more general statement.
\end{rem}

\section{Colouring graphs of bounded local density}\label{sec:proof}

Molloy and Reed's original proof of Theorem~\ref{thm:MoResparse} took the following strikingly basic, one might say na\"ive, form.
Given a graph $G$ and a palette, say, $[M]$ of $M$ colours, perform the following steps.
\begin{enumerate}
\item\label{naivestep1}
For each vertex $v\in V(G)$, independently and uniformly at random assign an element of $[M]$ as a colour for $v$.
\item\label{naivestep2}
For each edge of $G$ for which both endpoints have been assigned the same colour, remove the colour from one or both endpoints.
\item\label{naivestep3}
Complete the partial proper colouring to a full colouring of $G$ if possible.
\end{enumerate}
\noindent
Consult~\cite{MoRe02} for broader applications and further refinements of this method.

Molloy and Reed~\cite{MoRe97} showed that for $G$ a $\sigma$-sparse $\Delta$-regular graph and $M=(1-\eps)\Delta$, for some $\eps>0$ depending on $\sigma$, the above na\"ive strategy succeeds. In particular, their essential observation was that, after Step~\ref{naivestep2}, we obtain a partial proper colouring and expect in each neighbourhood for some colours to appear multiple times.
With the help of Talagrand's inequality and the Lov\'asz local lemma, they could moreover show that, with positive probability and uniformly for each vertex $v$, enough colours are repeated in $N(v)$ to ensure that the residual list of $v$ is larger than the residual degree of $v$. Step~\ref{naivestep3} is then carried out easily by a greedy colouring procedure. It is worth noting here that for Theorem~\ref{thm:MoResparse} it suffices by a standard reduction to consider $\Delta$-regular graphs.

Remarkably, despite the attention this result has received over the years, the only known methods for proving Theorem~\ref{thm:MoResparse} (including those of the present work) take the same basic form above.
We remark that in the conflict resolution of Step~\ref{naivestep2}, Molloy and Reed were nonpartisan and removed the colour from both endpoints of a monochromatic edge. The advance of Bruhn and Joos~\cite{BrJo18} was obtained by independently tossing a fair coin for each monochromatic edge to decide which of the two endpoints to uncolour in Step~\ref{naivestep2}, and then more accurately estimating the number of repeated colours in each neighbourhood via the inclusion-exclusion principle. This new estimate, however, did not satisfy the conditions of the usual combinatorial incarnations of Talagrand's inequality and so in order to fit it into the previous framework they devised their ``exceptional outcomes'' version of Talagrand's, Theorem~\ref{thm:upward certs}. 

The salient contribution of Bonamy {\em et al.}~\cite{BPP22} was to show that for $G$ a $\sigma$-sparse $\Delta$-regular graph, Step~\ref{naivestep3} could be performed after an iteration of Steps~\ref{naivestep1} and~\ref{naivestep2} upon the residual subgraph, and so on, until the residual list-sizes are larger than the maximum degree of the residual subgraph. The key for showing this to be possible was to demonstrate that neighbourhood intersections in the residual subgraph are expected to behave ``quasirandomly'' (see Subsections~\ref{sub:list colouring}), which helps to guarantee local sparsity (see Subsection~\ref{sub:proof}). Moreover, the list-sizes in the residual list-assignment are uniformly bounded from below in expectation and so, together with Theorem~\ref{thm:upward certs} and the Lov\'asz local lemma, we can be guaranteed a proper partial (list-)colouring of $G$ whose residual subgraph is contained in a $\sigma'$-sparse $\Delta'$-regular graph with a residual list-assignment having list-sizes all at least $(1-\eps')\Delta'$, where for $\eps$ sufficiently small we are guaranteed that $\eps'<\eps$.

Our work continues this evolution with two distinct but closely related alterations.

The first alteration is to improve in the conflict resolution of Step~\ref{naivestep2} by means of a priority assignment strategy.
We assign independently and uniformly at random to each vertex $v$ a {\em priority} $\pi(v) \in [0,1]$.
For each monochromatic edge $uv$ considered in Step~\ref{naivestep2}, we uncolour the vertex with lower priority, i.e.~if $\pi(u) \ge \pi(v)$ we uncolour $v$.
Curiously, this idea was already suggested at the end of Molloy and Reed's paper!
It is worth pointing out that Pemmaraju and Srinivasan~\cite{PeSr08} also applied this same strategy, albeit not iteratively, for other graph colouring problems. 

There are three important benefits to note regarding this first alteration. Solely for the purposes of explaining these benefits, we find it convenient to introduce two auxiliary digraphs $\GBJ$ and $\GHJK$. These encode Bruhn and Joos's conflict-resolution protocol (also used by Bonamy {\em et al.})~and ours, respectively, by having an arc $\overrightharp{uv}$ from $u$ to $v$ if $u$ uncolours $v$ whenever both are assigned the same colour in Step~\ref{naivestep1}.
Both digraphs have $G$ as the underlying undirected graph: $\GBJ$ orients each edge of $G$ independently and uniformly, and $\GHJK$ is an acyclic orientation of $G$ according to a uniform total ordering of $V(G)$.
First, notice there can be directed cycles in $\GBJ$, at least one vertex of which might be needlessly uncoloured, but not in $\GHJK$.
Second, given nonadjacent $v,w\in N(u)$, the events $(\overrightharp{uv}\in E(\GBJ))$ and $(\overrightharp{uw}\in E(\GBJ))$ are independent, whereas the events $(\overrightharp{uv}\in E(\GHJK))$ and $(\overrightharp{uw}\in E(\GHJK))$ are correlated. Thus if $|N(v)\cap N(w)|$ is large, then $v$ and $w$ are more likely to synchronise their colours.
Third, note that the in-degree of a vertex $v$ in $\GBJ$ has a $\mathrm{Bin}(\Delta,1/2)$ distribution, while in $\GHJK$ it has a $\mathrm{Unif}\{0,\dots,\Delta\}$ distribution. Thus, supposing $v$ has $x$ neighbours assigned the same colour as $v$ after Step~\ref{naivestep1}, in the Bruhn--Joos protocol $v$ keeps its colour with probability $1/2^x$, while in ours the probability is a much larger $1/x$.
It follows that to successfully colour in one iteration under the Bruhn--Joos protocol versus ours, one needs to maintain that $x$ is uniformly much smaller over all vertices, and the number $M$ of colours in the palette must therefore be larger.

We summarise these benefits of priority assignment as follows:
(1)~fewer needless uncolourings;
(2)~colour synchronisation for nonadjacent vertices with many common neighbours; and
(3)~robustness of the colouring procedure to using fewer colours.

Our second alteration critically takes advantage of this last benefit to improve on Steps~\ref{naivestep1} and~\ref{naivestep3}. For Step~\ref{naivestep1}, rather than using $M=(1-\eps)\Delta$ colours for some small but fixed $\eps>0$, we use $M=\lfloor\Delta/\gamma\rfloor$ colours for some large but fixed $\gamma$. This yields larger independent sets (which ultimately correspond to colour classes) and amplifies the colour synchronisation. Resolving conflicts using priorities we then adopt for Step~\ref{naivestep3} the iterative strategy of Bonamy {\em et al.}~to obtain a quasirandom residual subgraph and a residual list-assignment upon which we iterate. Denoting the maximum degree of this residual subgraph by $\Delta'$ we top up our lists so that each vertex has a list with $\lfloor\Delta'/\gamma\rfloor$ colours.  We iterate until we have a partial proper colouring such that the residual subgraph has negligible maximum degree relative to $\Delta$. (Loosely speaking, this procedure ``nibbles'' the original list of $(1-\eps)\Delta$ colours by at most $\lfloor\Delta/\gamma\rfloor$ colours at a time.) We may then greedily complete the colouring.

Informally, one can understand the procedure and its success as $\sigma\to 0$ by considering what happens if it is run with $M=1$, i.e.~with a single colour, say, red. 
The key question is, in $N(v)$ for $v\in V(G)$, if red appears (thus preventing $v$ from being red), how many times do we expect it to be repeated? Restricting our attention to $N(v)$, we estimate the number of repetitions after Step~\ref{naivestep2} by the expected number of red pairs less the expected number of red triples. We have exactly $\sigma\binom{\Delta}{2}$ potential pairs and by a result of Rivin \cite{Riv02}, at most $\sigma^{3/2}\binom{\Delta}{3}$ potential triples. The probability of any single vertex being red after conflicts are resolved is simply $1/\Delta$, the chance that it has higher priority than all of its neighbours,  while for nonadjacent pairs and triples it is \emph{roughly} $1/\Delta^2$ and $1/\Delta^3$ (ignoring correlations). This yields the estimate
\[\frac{1}{\Delta^2}\sigma\binom{\Delta}{2}-\frac{1}{\Delta^3}\sigma^{3/2}\binom{\Delta}{3}\sim \sigma\left(\frac12 - \frac{\sigma^{1/2}}{6}\right).\]
As we prove in Theorem~\ref{thm:independent}, this is in fact what we obtain by accurately accounting for the correlations.

\subsubsection*{Structure of the proof}

The engine of the proof of Theorem~\ref{col_result} is an effective method of randomly generating a large enough independent set of a $\sigma$-sparse $\Delta$-regular graph, using the conflict-resolution protocol described above. This is embodied in Theorem~\ref{thm:independent}, which we prove in Subsection~\ref{sub:generating independent set}.
In each step (or ``nibble'') of our colouring procedure, given a $\sigma$-sparse graph of maximum degree $\Delta$ with a $k$-list-assignment, we apply Theorem~\ref{thm:independent} and the Lov\'asz local lemma to obtain a partial proper list-colouring such that the residual subgraph is quasirandom, has maximum degree $\Delta'$, and the corresponding residual list-assignment contains a $k'$-list-assignment.
This ``nibble'' is embodied in Lemma~\ref{lem:list colouring} and is shown in Subsection~\ref{sub:list colouring}.
With Lemma~\ref{lem:list colouring} in hand, in Subsection~\ref{sub:proof} we are able to present in full the iterative colouring procedure and the proof of Theorem~\ref{col_result}.

\subsection{Sampling independent sets}\label{sub:generating independent set}

We introduce and analyse a randomised procedure to generate an independent set.
This procedure captures the behaviour of each colour class of the colouring procedure we analyse in Subsection~\ref{sub:list colouring}.
The idea is to assign to each vertex a random \emph{priority}
and resolve conflicting edges by removing the endpoint with lower priority.

Fix a parameter $\gamma>0$. Given a $\Delta$-regular graph $G=(V,E)$, the following procedure outputs a random independent set $\bI$ of $G$.
\begin{enumerate} 
\item\label{samplestep1} Activate each vertex of~$G$ with probability $\gamma/\Delta$,
  independently at random.
  Let $\A$ be the set of activated vertices.
\item\label{samplestep2} Assign to each activated vertex $v\in \A$ a number $\pi(v)$
  chosen uniformly at random in $[0,1]$.
\item\label{samplestep3} In order to resolve any conflict, i.e.~two neighbouring vertices in $\A$,
  remove the vertex with lower priority $\pi$.
  This yields the independent set
  \[\bI=\sst{v\in \A}{\pi(v) > \pi(u) \text{ for every }u \in N(v)\cap \A},\]
  consisting of all the local maxima of $\pi$ in $G[\A]$.
\end{enumerate}

The purpose of this section is to prove the following result.
\begin{theorem}\label{thm:independent}
  For every $\iota>0$, there are $\Delta_{\ref{thm:independent}}=\Delta_{\ref{thm:independent}}(\iota)$
  and $\gamma_{\ref{thm:independent}}=\gamma_{\ref{thm:independent}}(\iota)$ such that the following holds.
  Let $G$ be a $\sigma$-sparse $\Delta$-regular graph with $\Delta\geq\Delta_{\ref{thm:independent}}$,
  and let
  $\bI$ be a random independent set obtained by the algorithm above with some parameter
  $\gamma\geq\gamma_{\ref{thm:independent}}$.
  For every vertex $r\in V(G)$, 
\[
\left|\Prob[r\in \bI]-\frac{1-e^{-\gamma}}{\Delta}\right| \le \frac{2}{\Delta^2}.
\]
Moreover, setting $\bI_r = N(r) \cap \bI$, it holds that
  \[
  \frac{\Prob\left[\bI_r \neq \varnothing\right]}{\E\left[|\bI_r|\right]}
 \leq 
  1-\eps_{\ref{col_result}}(\sigma)+\iota.
  \]
\end{theorem}
\noindent
The ratio in Theorem~\ref{thm:independent} can be read as the inverse of the average size of~$\bI_v$
when $\bI_v$ is nonempty.
For comparison, if $\bI$ is chosen instead as a random colour class of a proper $\chi$-colouring of~$G$, then this ratio is a lower bound for $\chi/\Delta$.
After proving Theorem~\ref{thm:independent}, we use the remainder of the section to transfer this bound to the chromatic number, showing that
$\chi(G)\leq (1-\eps_{\ref{col_result}}(\sigma)+o(1))\cdot\Delta$.

In the proof, we say that a vertex $u$ \emph{trumps} a vertex $v$ if $uv$ is an edge,
$u$ and $v$ are activated and $\pi(u)\ge\pi(v)$.
With this vocabulary, $\bI$ is the set of activated vertices that are not trumped.
\begin{proof}[Proof of Theorem~\ref{thm:independent}]
It is convenient to instead prove a slightly stronger, local version of Theorem~\ref{thm:independent},
where $\sigma$ is the local sparsity of $r$, so that $G[N(r)]$
contains exactly $(1-\sigma)\binom{\Delta}{2}$ edges.
Proving the local version is enough because
$\eps_{\ref{col_result}}(\sigma)$ is an increasing function of $\sigma$.

For convenience, we define
\[
 \I_n=\sst{\{u_i\}_{i=1}^n\subseteq N(r)}{\forall i,j\in\{1,\dots,n\}, ~u_iu_j\not\in E(G)}
\]
 as the collection of independent sets of size $n$ in $N(r)$.
 For brevity, we often write, for example, $uv\in \mathcal{I}_2$ instead of $\{u,v\}\in \mathcal{I}_2$ or $uvw\in \mathcal{I}_3$ instead of $\{u,v,w\}\in \mathcal{I}_3$.
For an independent set
$\{u_i\}_{i=1}^n\in \mathcal{I}_n$, we define
\[
 \Pk(u_1,\dots,u_n):=\Prob\left[\forall i\in[n],~u_i\in \bI \,\,\middle|\,\, \forall i\in[n],~u_i \in \A\right].
\]
 This is the probability that all of the considered vertices
are retained in the independent set if they were activated in Step~\ref{samplestep1}.

Let us show the first part of the theorem.
\begin{claim}\label{claim:v in I}
        For every vertex $v\in V$,
        it holds that 
\[
\Prob[v\in\bI]=
        \frac{1}{\Delta}\int_0^\gamma\left(1 -\frac{x}{\Delta}\right)^{\Delta} dx\in
\left[\frac{1-e^{-\gamma}}{\Delta}-\frac{2}{\Delta^2},\frac{1-e^{-\gamma}}{\Delta}\right].
\]
\end{claim}
\begin{proof}
Assuming that~$v$ is activated
and given $\pi(v)$, the probability that $v$ is trumped by some
other vertex $q\in N(v)$ is the probability~$q$ is activated times the probability
that $\pi(q)\geq\pi(v)$.
This latter probability is $1-\pi(v)$ because $\pi(q)$ is chosen uniformly at random in $[0,1]$.
Consequently,
\[
        \Prob[q\text{ trumps }v]=\frac{\gamma}{\Delta}(1-\pi(v)) = \frac{\gamma x}{\Delta},
\]
where we write $x=1-\pi(v)$ in order to simplify integration.
Expressing the probability that no neighbour of~$v$ trumps~$v$ as
$\left(1- \gamma x/\Delta\right)^\Delta$ and integrating over the possible values of $x$,
we get
\[
        \Prob[v\in\bI] =
        \Prob[v\in\A]\cdot\int_0^1\left(1 -\frac{\gamma x}{\Delta}\right)^{\Delta} dx
        = \frac{\gamma}{\Delta}\int_0^1\left(1 -\frac{\gamma x}{\Delta}\right)^{\Delta} dx
        = \frac{1}{\Delta}\int_0^\gamma\left(1 -\frac{x'}{\Delta}\right)^{\Delta} dx'
\]
which proves the first part of the claim.
This value is always less than the limit $(1-e^{-\gamma})/\Delta$
since
\[\int_0^\gamma\left(1-\frac{x}{\Delta}\right)^{\Delta}dx\leq\int_0^\gamma e^{-x}dx=1-e^{-\gamma}.
\]
For the lower bound, we use the fact $e^{-t}(1-t^2)\leq 1-t$ for every $t\in[0,1]$
applied to $t=x/\Delta$.
This gives
\[\int_0^\gamma\left(1-\frac{x}{\Delta}\right)^{\Delta}dx\geq \int_0^\gamma e^{-x}\left(1-\frac{x^2}{\Delta^2}\right)^{\Delta}dx
\geq \int_0^\gamma e^{-x}dx -\frac{1}{\Delta}\int_0^\gamma e^{-x}x^2 dx.
\]
Last, bounding the second integral by $\int_0^\infty e^{-x}x^2dx=2$,
we deduce that $\Prob[v\in\bI]\geq (1-e^{-\gamma})/\Delta-2/\Delta^2$. 
\end{proof}
As a consequence of Claim~\ref{claim:v in I},
the expected size of $\bI_r=\bI\cap N(r)$ is
\[
\E[|\bI_r|]=\sum_{v\in N(r)}\Prob[v\in \bI]= 1 - e^{-\gamma} + o(1).
\]


We now prove the following reduction.
\begin{claim}\label{claim:no outside p3}
We may assume that no pair of distinct vertices in $N(r)$ have a common neighbour outside of
$N[r]$.
\end{claim}
\begin{proof}
Assume otherwise that there is vertex $w\in V(G)\setminus N[r]$
with at least two distinct neighbours in $N(r)$.
We construct a $\sigma$-sparse $\Delta$-regular graph $G'$ as the disjoint union
of $G\setminus w$ and a complete bipartite graph $K_{\Delta-1,\Delta}$ on a vertex partition $X\cup Y$ with $|X|=\Delta$ and $|Y|=\Delta-1$,
in which we further connect each vertex in $N_G(w)$ to a distinct vertex of~$X$.
Here the role of the bipartite graph is only to preserve the $\Delta$-regularity.
The neighbourhood of~$r$ is the same in~$G$ and in~$G'$, so $N_{G'}(r)$ induces
exactly $(1-\sigma)\binom{\Delta}{2}$ edges in~$G'$.
Let $\bI'$ be the random independent set obtained by our procedure on~$G'$
and consider $\bI_r'=\bI \cap\{N_{G'}(r)\}$.
We know as a corollary of Claim~\ref{claim:v in I}
that $\E[|\bI_r|]=\E[|\bI_r'|]$.
Further, we claim that
\begin{equation}\label{eq:G' is worse}
\Prob[\bI_r\neq\varnothing]\leq\Prob[\bI_r'\neq\varnothing].
\end{equation}
In that case,
$\Prob[\bI_r\neq\varnothing]/\E[|\bI_r|]\leq \Prob[\bI_r'\neq\varnothing]/\E[|\bI_r'|]$,
so it is enough to prove Theorem~\ref{thm:independent} for $G'$ and $r$ to deduce it for $G$ and $r$.
Since the number of common neighbours of $N(r)$ outside $N[r]$ is strictly
lower in $G'$ than in $G$,
the iteration of this transformation terminates, which proves Claim~\ref{claim:no outside p3}.

It remains to show~\eqref{eq:G' is worse}.
To do so, we couple $\bI$ and $\bI'$
in such a way that~$\bI_r$ is empty in every outcome where~$\bI_r'$ is.
Assuming that $\A$ and $\pi$ are given, the activation set~$\A'$
and the priority function~$\pi'$ are defined as follows.
First, vertices of $V(G)\cap V(G')=V(G)\setminus\{r\}$ are activated accordingly for $\A$
and $\A'$ and get the same priority, that is
$\A'\cap V(G):= \A\setminus\{r\}$ and $\pi'=\pi$ on this set.
Now, let $U$ be the set of vertices of $\A\cap N(r)\cap N(w)$
that trump all their neighbours in $\A\setminus \{w\}$ for $G$,
i.e. the set of vertices of $N(r)\cap N(w)$ that would be in $\bI$ if we ignore $w$
in Step~\ref{samplestep3}.
If $U\neq\varnothing$,
let $v_U$ be the vertex of $U$ with the highest value $\pi(v_U)$
and let~$w'$ be the unique neighbour of~$v_U$ in~$X$ for the graph~$G'$.
We activate $w'$ for $\A'$ if $w$ is activated for $\A$,
and in this case we set $\pi'(w'):=\pi(w)$.
Next, we activate (for $\A'$) the remaining vertices of $X\cup Y$ independently at random with
probability $\gamma/\Delta$ and we give them priorities $\pi'$ chosen independently,
uniformly at random in $[0,1]$.
The set $\bI'$ is then defined from $\A'$ and $\pi'$ similarly as for $\bI$,
as the set of vertices of $\A'$ whose value by $\pi'$ is larger than all of their neighbours
in $\A'$.

It is clear that $\A'$, $\pi'$, and further $\bI$ are distributed as in our procedure
because these activations and priorities are mutually independent.
Crucially, note that the choice of~$v_U$ and~$w'$ (when $U\neq\varnothing$)
is independent from the value of $\pi(w)$.
If $U=\varnothing$, then $\bI_r=\varnothing=\bI_r'$ regardless of
the state of $w$ and the vertices of $X\cup Y$,
so assume that $U\neq\varnothing$ and $\bI_r'=\varnothing$.
Since $\bI_r$ and $\bI_r'$ coincide on $N(r)\setminus N(w)$,
we know that $\bI_r\subseteq U$.
Further, since $v_U$ is in $U$ but not in $\bI_r'$, this vertex is trumped by $w'$,
so $w'\in\A'$ and $\pi(v_U)=\pi'(v_U)\le\pi'(w')=\pi(w)$. 
Moreover, $\pi(v_U)$ is by definition of $v_U$ the highest value of $\pi(U\cap\A)$,
so $w$ trumps every activated vertex of $U$, and further $\bI_r=\varnothing$.
\end{proof}

It remains to estimate $\Prob[\bI_r\neq\varnothing]$.
To do so, let $P_r$ be the the number of pairs
in $N(r)\cap \bI_r$, i.e. $P_r=\binom{|\bI_r|}{2}$, and let $T_r$ be the number of triples
in $N(r)\cap \bI_r$, i.e. $T_r=\binom{|\bI_r|}{3}$.
Our bound on $\Prob[\bI_r\neq\varnothing]$ relies on the following variation
of the inclusion-exclusion principle:
\begin{equation}\label{eq:inclusion-exclusion} 
\Prob[\bI_r \neq \varnothing] \leq \E[|\bI_r|] - \E[P_r] + \E[T_r].
\end{equation}
To see this, fix $\bI_r$ and
apply the inclusion-exclusion principle to $|\bI_r|$ identical sets of size $1$
to get 
\[\mathbb{1}_{\bI_r \neq \varnothing}\leq |\bI_r| - \binom{|\bI_r|}{2} + \binom{|\bI_r|}{3},\]
where $\mathbb{1}_{\bI_r \neq \varnothing}$ equal $1$ if $\bI_r$ is nonempty and $0$ otherwise.
Taking the expectation of this last inequality proves~\eqref{eq:inclusion-exclusion}.

Let us now estimate $P_r$ and $T_r$ in terms of the following parameter.
Given a pair $uv\in \I_2$,
define 
\[\ell_{uv}:=\frac{1}{\Delta}|N(v) \cap N(u)|.\]
We start with $P_r$.
Note that here and later, we use the terminology $\Eon{s\in S}$ to denote an expectation where $s$ is chosen uniformly from the set $S$.
\begin{claim}\label{claim:pairs}
It holds that
\[
  \E[P_r] = \sigma\cdot\Eon{uv\in \I_2}\left[\frac{1}{(1-\ell_{uv})(2-\ell_{uv})}\right]
  + o_\gamma(1)
  + o_\Delta(1).
\]
\end{claim}
\begin{proof}
First note that
\[
   \E[P_r]= \sum_{uv\in \I_2}\Pk(u,v)\cdot\Prob[u, v \in \A].
\]
As the activation of the vertices are independent we have
$\Prob[u, v\in \A]=\gamma^2/\Delta^2$
so let us consider $\Pk(u,v)$.

Assume first that $u$ and $v$ are activated and that $\pi(u)\geq\pi(v)$.
In that case, a vertex of $N(u)\cap N(v)$ that trumps $u$ necessarily trumps $v$,
so $u$ and $v$ are in $\I$ exactly when $v$ trumps the vertices of~$N(v)$
and $u$ trumps the vertices of $N(u)\setminus N(v)$.
As a consequence, the probability that both~$u$ and~$v$ are in $\bI$ is
\[
        \left(1-\frac{\gamma}{\Delta}(1-\pi(v))\right)^\Delta\cdot
        \left(1-\frac{\gamma}{\Delta}(1-\pi(u))\right)^{(1-\ell_{uv})\Delta}
        = e^{-\gamma(1-\pi(v))}e^{-\gamma(1-\ell_{uv})(1-\pi(u))}+o_\Delta(1),
\]
where we again write the probability that a vertex $w$ trumps
an activated neighbour~$t$ as $\frac{\gamma}{\Delta}(1-\pi(t))$.

Integrating on the values of $x=1-\pi(u)$ and $y=1-\pi(v)$, we get
\[
   \Prob[u,v\in\bI\text{ and }\pi(u)>\pi(v)| u,v\in\A]=
   \int_0^1\int_0^xe^{-\gamma y}e^{-\gamma(1-\ell_{uv})x} dydx + o_\Delta(1).
\]
Accounting for the symmetry between the case $\pi(u)>\pi(v)$ and $\pi(u)<\pi(v)$,
we deduce
\[
    \Pk(u,v)=2\int_0^1\int_0^xe^{-\gamma y}e^{-\gamma(1-\ell_{uv})x} dydx + o_\Delta(1),
\]
where the dependence is as $\Delta\to\infty$.

Computing the integral, we obtain
%
%
\begin{align*}
  \int_0^1\int_0^x\mathrm{e}^{-\gamma y}\mathrm{e}^{-\gamma \left(1-\ell_{uv}\right)x} dydx 
  &=
  \frac{1}{\gamma}\int_0^1\left(1-e^{-\gamma x}\right)\mathrm{e}^{-\gamma(1-\ell_{uv})x} dx \\
  &=
  \frac{1}{\gamma^2}\left(\frac{1-e^{-\gamma(1-\ell_{uv})}}{1-\ell_{uv}}-\frac{1-e^{-\gamma(2-\ell_{uv})}}{2-\ell_{uv}}
\right)\\
  &= \frac{1}{\gamma^2(1-\ell_{uv})(2-\ell_{uv})} + o\left(\frac{1}{\gamma^2}\right).
\end{align*}

Recalling that $|\I_2|=(\sigma/2+o(1))\Delta^2$,
we can conclude that
\begin{align*}
  \E[P_r] & =
  \sum_{uv\in \I_2}\frac{\gamma^2}{\Delta^2}\cdot\frac{2}{\gamma^2(1-\ell_{uv})(2-\ell_{uv})}
  +o_\gamma(1)+o_\Delta(1)\\
  &=\sigma\cdot\Eon{uv\in\I_2}\left[\frac{1}{(1-\ell_{uv})(2-\ell_{uv})}\right]+o_\gamma(1)+o_\Delta(1).\qedhere
\end{align*}
\end{proof}

Let us estimate the number of triples in $\bI_r$.
For that purpose, we first derive an integral expression for $\E[T_r]$.
Set
\[
   f(\ell_1,\ell_2,\ell_3)=
   \frac{1}{(2-\ell_1)(3-\ell_1-\ell_2-\ell_3)}.
\]
We now prove the following expression.
\begin{claim}\label{claim:pkeep triples}
For every triple $uvw\in\I_3$,
\[
   \Pk(u,v,w)\leq
   \frac{2}{\gamma^3}\cdot
   \left[
   f(\luv,\luw,\lvw)+
   f(\lvw,\luv,\luw)+
   f(\luw,\lvw,\luv)
   \right].
   \]
\end{claim}
\begin{proof}
Assuming that $u$, $v$ and~$w$ are activated and that
the values $x=1-\pi(u)$, $y=1-\pi(v)$ and~$z=1-\pi(w)$
are given, we aim to express the probability that $\{u,v,w\}\subseteq\bI$.
If moreover $x\ge y\ge z$, i.e.~$\pi(u)\le\pi(v)\le\pi(w)$,
this last event happens exactly when 
none of the vertices of $N(u)$ trump~$u$,
none of the vertices of $N(v)\setminus N(u)$ trump~$v$ and
none of the vertices of $N(w)\setminus(N(u)\cup N(v))$ trump~$w$.
The sizes of these sets are estimated by respectively
$|N(u)|=\Delta$, $|N(v)\setminus N(u)|=(1-\luv)\Delta$
and $|N(w)\setminus(N(u)\cup N(v))|\geq(1-\luw-\lvw)\Delta$.
The probability that $\{u,v,w\}\subseteq\bI$ in this case is therefore at most
 \[
        \left(1-\frac{\gamma}{\Delta}x\right)^\Delta
        \left(1-\frac{\gamma}{\Delta}y\right)^{\Delta(1-\luv)}
        \left(1-\frac{\gamma}{\Delta}z\right)^{\Delta(1-\luw-\lvw)},
 \]
which is bounded from above by its limit
 \[
   p_{xyz}:=
    e^{-\gamma\left(x +
    \left(1-\luv\right)y
    + (1-\luw-\lvw)z
    \right)}.
  \]
Integrating over the possible values of $x$, $y$ and $z$
satisfying $1\geq x\geq y\geq z\geq 0$ gives
\begin{align*}
   \int_0^1\int_{z}^1\int_{y}^1p_{xyz}dxdydz
    &=\frac{1}{\gamma^3}\cdot
    \int_{0}^\gamma\int_{z'}^\gamma\int_{y'}^\gamma
    e^{-x'-(1-\luv)y'
    - (1-\luw-\lvw)z'}dx'dy'dz'\\
    &\leq \frac{1}{\gamma^3}\cdot
    \int_{0}^\infty\int_{z}^\infty\int_{y}^\infty
    e^{-x-(1-\luv)y
    - (1-\luw-\lvw)z}dxdydz\\
    &= \frac{1}{\gamma^3}\cdot\int_{0}^\infty\int_{z}^\infty e^{-(2-\luv)y- (1-\luw-\lvw)z}dydz \\
    &= \frac{1}{\gamma^3}\cdot\frac{1}{2-\luv}\cdot\int_{0}^\infty e^{- (3-\luv-\luw-\lvw)z}dz\\
    &= \frac{1}{\gamma^3}\cdot f(\luv,\luw,\lvw),
\end{align*}
where the first line comes from the change of variables $x'=\gamma x$, $y'=\gamma y$
and $z'=\gamma z$.
To summarise, we have
\[
\Prob\left[u,v,w\in\bI\text{ and }\pi(u)\le\pi(v)\le\pi(w)\,\,\middle|\,\,u,v,w\in\A\right]
\leq \frac{1}{\gamma^3}f(\luv,\luw,\lvw).
\]
Taking into account the six possible orderings of $\pi(u)$, $\pi(v)$ and~$\pi(w)$
and the symmetry of $f$ between the second and the third variable yields
\begin{align*}
   \Pk(u,v,w)\leq&\frac{1}{\gamma^3}(
   2f(\luv,\luw,\lvw)+
   2f(\lvw,\luv,\luw)+
   2f(\luw,\lvw,\luv)
),
\end{align*}
which finishes the proof of the claim.
\end{proof}
The function $f$ satisfies the following bound that isolates its parameters:
\begin{equation}\label{eq:f as sum}
f(\ell_1,\ell_2,\ell_3)+f(\ell_3,\ell_1,\ell_2)+f(\ell_2,\ell_3,\ell_1) \leq
        \frac13\cdot\sum_{i=1}^3\frac{1}{(2-\ell_i)(1-\ell_i)}
\end{equation}
for every $\ell_1,\ell_2,\ell_3\in[0,1]$.
This relation can be proven as follows:
\begin{align*}
f(\ell_1,\ell_2,\ell_3)+f(\ell_3,\ell_1,\ell_2)+f(\ell_2,\ell_3,\ell_1)&=
\frac{1}{3-\ell_1-\ell_2-\ell_3} \cdot \sum_{i=1}^3\frac{1}{2-\ell_i}\\
&\leq \frac13\cdot\sum_{i=1}^3\frac{1}{3-3\ell_i} \cdot \sum_{i=1}^3\frac{1}{2-\ell_i}\\
&\leq \sum_{i=1}^3\frac{1}{3-3\ell_i}\cdot\frac{1}{2-\ell_i}.
\end{align*}
Here, the second line is obtained by convexity of the function $x\mapsto\frac{1}{3-x}$
applied on the left factor and the last step is an application of Chebyshev's sum inequality.

We deduce, from Claim~\ref{claim:pkeep triples} and~\eqref{eq:f as sum},
the following bound on $\E[T_r]$.
\begin{claim}\label{claim:triples}
It holds that
\[\E[T_r]\leq \frac{|\I_3|}{\Delta^3}+
\frac\sigma6\cdot\Eon{uv\in \I_2}\left(
\frac{2}{2-\ell_{uv}}+\luv-1
\right)
+o_\Delta(1).\]
\end{claim}
\begin{proof}
We compute the expected number of triples similarly as for the pairs:
\[\E[T_r]=\sum_{uvw\in\I_3}\Prob[u,v,w\in\A]\cdot\Pk(u,v,w)=\frac{\gamma^3}{\Delta^3}\cdot\sum_{uvw\in\I_3}\Pk(u,v,w).
\]
Further, applying Claim~\ref{claim:pkeep triples} and Equation~\eqref{eq:f as sum} gives
\[
   \E[T_r]
   \leq\frac{1}{3\Delta^3}\cdot\sum_{uvw\in \I_3}\sum_{ab\in\{uv,uw,vw\}}\frac{2}{(2-\ell_{ab})(1-\ell_{ab})}.
\]
We decompose this expression into two parts:
   \begin{equation}\label{eq:Tr decomposed}
   \E[T_r]\leq\frac{|\I_3|}{\Delta^3}+
   \frac{1}{3\Delta^3}\cdot
   \sum_{uvw\in \I_3}\sum_{ab\in\{uv,uw,vw\}}\left(\frac{2}{(2-\ell_{ab})(1-\ell_{ab})}-1\right).
   \end{equation}
   Consider the term 
   \[R:=\sum_{uvw\in \I_3}\sum_{ab\in\{uv,uw,vw\}}\left(\frac{2}{(2-\ell_{ab})(1-\ell_{ab})}-1\right).\]
   A pair $uv\in \I_2$ contributes to this double sum once for each $w\in N(r)$
   such that $uvw\in\I_3$.
   Since such a vertex $w$ cannot be a neighbour of $v$ and, by Claim~\ref{claim:no outside p3},
   each of the $\luv\Delta$ common neighbours of $u$ and $v$ are in $N[r]$,
   we know that there are
   at most $(1-\luv)\Delta$ such vertices $w$. It follows that
   \[
   R\leq \sum_{uv\in\I_2}(1-\luv)\Delta\cdot\left(\frac{2}{(2-\luv)(1-\luv)}-1\right)
   =\Delta\cdot\sum_{uv\in\I_2}\left(\frac{2}{2-\luv}+\luv-1\right).
   \]
   It remains to write the sum as an expectation using $|\I_2|=(\sigma/2+o(1))\Delta^2$:
   \begin{equation}\label{eq:R}
   \frac{R}{\Delta^3}\leq
   \frac{\sigma}{2}\cdot\Eon{uv\in\I_2}\left(\frac{2}{2-\luv}+\luv-1\right)
   + o_\Delta(1).
   \end{equation}
   The claim then follows from Equations~\eqref{eq:Tr decomposed} and~\eqref{eq:R}.
\end{proof}
We are now ready to conclude the proof of the theorem.
By Claims~\ref{claim:pairs} and~\ref{claim:triples},
\begin{align*}
        \E[P_r-T_r] &\geq
\sigma\cdot\Eon{uv\in \I_2}\left[\frac{1}{(1-\ell_{uv})(2-\ell_{uv})}\right]
-\left(\frac{|\I_3|}{\Delta^3}+
\frac\sigma6\cdot\Eon{uv\in \I_2}\left[
\frac{2}{2-\ell_{uv}}+\luv-1
\right]\right)+o_\Delta(1)+o_\gamma(1)\\
&=
\frac\sigma6\cdot\Eon{uv\in \I_2}\left[
\frac{2}{1-\luv}+1-\luv
\right]
-\frac{|\I_3|}{\Delta^3}+o_\Delta(1)+o_\gamma(1).
\end{align*}
The function $g:x\mapsto \frac{2}{1-x}+1-x$ is increasing on $[0,1]$,
so we may bound the last expectation by $g(0)=3$.
Further, a theorem from Rivin~\cite{Riv02} shows that $|\I_3|\leq\sigma^{3/2}\binom{\Delta}{3}$.
It follows that
\begin{equation}\label{eq:pairs - triples}
\E[P_r-T_r] \geq \frac{\sigma}{2}-\frac{\sigma^{3/2}}{6}+o_\Delta(1)+o_\gamma(1)
= \eps_{\ref{col_result}}(\sigma)+o_\Delta(1)+o_\gamma(1).
\end{equation}
Recall that as a consequence to Claim~\ref{claim:v in I},
the expected size of $\bI_r$ is $1-e^{-\gamma}+o_\Delta(1)$.
Using Equation~\eqref{eq:inclusion-exclusion}, we conclude that
\[
\frac{\Prob[\bI_r\neq\varnothing]}{\E[|\bI_r|]}\leq
\frac{\E[|\bI_r|-P_r+T_r]}{\E[\bI_r]}\leq
\frac{1-e^{-\gamma}-\eps_{\ref{col_result}}(\sigma)}{1-e^{-\gamma}}+o_\Delta(1)
=1-\eps_{\ref{col_result}}(\sigma)+o_\Delta(1)+o_\gamma(1),
\]
which proves the theorem.
\end{proof}

\begin{rem}
    As observed at the beginning of the section, the value of $\eps_{\ref{col_result}}$ is what we would expect to obtain if we activated the entire graph (used a single colour) and there were no correlations (created by common neighbours). Ultimately we proved above that small lists behave like lists of size 1 and that, given there are many triples, correlations help pairs more than they help triples, thus improving the colouring. 
\end{rem}

\subsection{One nibble}\label{sub:list colouring}

Our strategy is to build up the colouring gradually through a sequence of partial proper (list-)colourings. Each step of this sequence has the sampling procedure of Subsection~\ref{sub:generating independent set} at its core. In Lemma~\ref{lem:list colouring} below, we prove (with the help of the Lov\'asz local lemma and Theorem~\ref{thm:upward certs}) that the sampling, with positive probability, has suitable properties for continuation of the iterative colouring procedure.

In addition to the notation for partial list colouring given in Subsection~\ref{sub:notation}, we need the following definition, after~\cite{BPP22}.
Given a graph $G$, a vertex subset $A\subseteq V(G)$, and $\mu>0$, we say that $G[A]$ is a {\em $\mu$-quasirandom} subgraph of $G$ if for all not necessarily distinct vertices $u,v\in A$,
\[
\left||N(u)\cap N(v) \cap A|-\mu|N(u)\cap N(v)|\right|\le \sqrt{\Delta}\log^5\Delta.
\]
Note that in the special case $u=v\in A$, the condition specialises to $|d_{G[A]}(u)-\mu d_G(u)| \le \sqrt{\Delta}\log^5\Delta$.
 
\begin{lemma}\label{lem:list colouring}
  For every $\iota>0$ and
  $\gamma\ge\gamma_{\ref{lem:list colouring}}(\iota)=\gamma_{\ref{thm:independent}}(\iota/2)$,
  there is $\Delta_{\ref{lem:list colouring}}=\Delta_{\ref{lem:list colouring}}(\iota,\gamma)$
  such that the following holds.
  Fix some parameters $\sigma>0$ and $\Delta \ge \Delta_{\ref{lem:list colouring}}$
  and set $k=\lceil\Delta/\gamma\rceil$ and $\mu = 1-(1-e^{-\gamma})/\gamma$.
  For every $\sigma$-sparse graph~$G$ with maximum degree~$\Delta$
  and a given $k$-list-assignment $L$, there is a partial proper
  $L$-colouring $c$ such that the residual subgraph $G_c$ is a $\mu$-quasirandom subgraph.
  Furthermore, in the residual list-assignment $L_c$, writing $\Delta'$ for the maximum degree of $G_c$, each list has size at least $k'$ where
  \[
  \left|\frac{k-k'}{\Delta - \Delta'}
  - (1 - \eps_{\ref{col_result}}(\sigma))\right|
  \leq \iota.
  \]
 \end{lemma}
\begin{proof}
  \resetclaimcount
  Set $G=(V,E)$.
In order to have each colour class behaving similarly, we first make $G$ ``colour-wise regular''
using the following statement.
\begin{claim}\label{claim:colour regular}
  There exists $V_1\supseteq V$, a graph $G_1=(V_1,E_1)$ that contains $G=G_1[V]$ as an induced subgraph,
  and a $k$-list-assignment $L_1$ of $G_1$ that extends $L$ (i.e.~$L_1|_{V}=L$)
  such that for every colour $c\in \bigcup_{v\in V_1}L(v)$,
  the subgraph of $G_1$ induced by those vertices~$v$ such that $L(v) \ni c$
  is $\Delta$-regular and $\sigma$-sparse.
\end{claim}
\begin{proof}
  As a first step,
  we embed $G$ in a $\Delta$-regular $\sigma$-sparse graph $H$ by the following iterative process.
  Start with $H=G$ and as long as $H$ is not $\Delta$-regular, take a copy $H'$ of $H$
  and add an edge between a vertex~$v$ in~$H$ and its copy in~$H'$ if $d_H(v)<\Delta$.
  Then set $H:=H\cup H'$ and repeat.
  As each step of this construction does not increase the number of triangles that contain
  each vertex, the graph $H$ is $\sigma$-sparse.
  By giving to each copy of a vertex~$v$ a copy of the list $L(v)$, we can extend $L$ to a $k$-list-assignment of $H$.
  
  Let $N$ be the smallest multiple of $k$ such that $L(v) \subseteq [N]$ for every vertex
  $v \in V$.
  Let us define the graph $G_1=(V_1,E_1)$ as a $N/k$-blow-up of $H$.
  Specifically, $V_1$ contains $N/k$ copies $v_1,\dots,v_{N/k}$
  of each vertex $v$ of~$V(H)$,
  and $u_iv_j\in E_1$ whenever $uv\in E$.
  For every $v\in V(H)$, we take the convention that $v=v_1$, so that $G_1[V]=G$.
  For every vertex $v$ of $V(H)$,
  we define the lists $L(v_1),\dots,L(v_{N/k})$ so that they form a partition of~$[N]$
  and $L(v_1)=L(v)$.

  Note that every colour $c\in[N]$ is contained in exactly one~$L(v_i)$ for each vertex~$v\in V(H)$.
  Consequently, the subgraph of~$G$
  induced by $\sst{v_i\in V_1}{ c\in L(v_i})$ is isomorphic to $H$,
  and therefore is $\Delta$-regular and $\sigma$-sparse.
\end{proof}
We now apply the following random colouring procedure to the graph $G_1$.
First, assign a colour $c_0(v)$ chosen uniformly at random from
the list~$L(v)$ to each vertex~$v\in V_1$.

This yields a (not necessarily proper) colouring $c_0:V_1\to\N$.
In order to resolve conflicts, we use the method discussed in Subsection~\ref{sub:generating independent set}:
independently assign to each vertex~$v\in V_1$ a priority~$\pi(v)$ chosen uniformly at random in $[0,1]$,
and define a set of \emph{uncoloured vertices} as
\[
U=\sst{v\in V_1}{\exists u\in N_{G_1}(v), ~c_0(u)=c_0(v)\text{ and }\pi(u)\ge\pi(v)}.
\]
As intended, the restriction of $c_0$ to $V_1\setminus U$,
the set of vertices that are not uncoloured,
is a partial proper colouring of~$G_1$.
Let $c$ be the partial proper colouring of~$G$ obtained by further restricting~$c_0$
to $V\setminus U$.
The residual subgraph of~$G$ with respect to~$c$ is therefore the induced subgraph $G_c=G_1[V\cap U]$.
Let~$\Delta'$ be the maximum degree of $G_c$, as in the statement of the theorem.
For every vertex~$v\in V\cap U$,
we write $\del(v)=c(N_G(v)\setminus U)$
for the set of colours used by neighbours of~$v$ under $c$.
The residual list-assignment of $G_c$ is defined as
$L_c(v)=L(v)\setminus \del(v)$ for every $v\in V\cap U$.
Further, define $k'= k-(1-\mu)(1-\eps_{\ref{col_result}}(\sigma)+\iota/2)\Delta - \sqrt{\Delta}\log^2\Delta$.
We aim to prove that $c$, $G_c$, $L_c$, $k'$ and~$\Delta'$ satisfy the theorem with positive
probability.

In order to have $|L_c(v)|\geq k'$,
it suffices to prevent the following ``bad'' event for every $v\in V\cap U$:
\begin{equation}\label{eq:bad1}
  |\del(v)| < k-k'
  \tag{$B_v$}.
\end{equation}
In order to ensure that $G'$ is a $\mu$-quasirandom subgraph of~$G$, we need to prevent
the following event~$B_{u,v}$
for every $u,v\in V(G)\cap U$:
\begin{equation}\label{eq:bad2}
\left||N_G(u)\cap N_G(v)\cap U|-\mu|N_G(u)\cap N_G(v)| \right|> \sqrt\Delta\log^5\Delta
.
\tag{$B_{u,v}$}
\end{equation}

Let us prove that preventing~\eqref{eq:bad1} and~$(B_{v,v})$ 
for every $v\in V\cap U$ is enough to ensure the claimed
bound on the ratio $(k-k')/(\Delta-\Delta')$.
Indeed, if $(B_{v,v})$ does not hold, then
\[|d_{G_c}(v)-\mu d_{G}(v)|=\left||N_{G}(v)\cap U|- \mu |N_G(v)|\right|\le \sqrt{\Delta}\log^5\Delta,\]
so in particular $d_{G_c}(v)=\mu\Delta + O(\sqrt{\Delta}\log^5\Delta)$
for every $v\in V\cap U$,
so $\Delta'=\mu\Delta + O(\sqrt{\Delta}\log^5\Delta)$.
It follows from this and the definition of~$k'$ that
\[
\frac{k-k'}{\Delta-\Delta'} =
\frac{(1-\mu)(1-\eps_{\ref{col_result}}(\sigma)+\iota/2)\Delta+ \sqrt\Delta\log^2\Delta}{\Delta-\mu\Delta + O(\sqrt{\Delta}\log^5\Delta)}= 1-\eps_{\ref{col_result}}(\sigma)+\frac\iota2+ O(\Delta^{-1/2}\log^5\Delta).
\]
If $\Delta_{\ref{lem:list colouring}}$ is large enough, then the
$O(\Delta^{-1/2}\log^5\Delta)$ term above is smaller than~$\iota/2$
for every $\Delta\ge\Delta_{\ref{lem:list colouring}}$.

Given a colour~$a$,
denote~$V_{a}$ the set of vertices of~$V_1$ whose list contains~$a$.
Recall that by Claim~\ref{claim:colour regular},
the graph $G_1[V_{a}]$ is $\sigma$-sparse and $\Delta$-regular.
The key property of the random colouring $c_0$ is that the independent set
$\bI_a=c_0^{-1}(\{a\})\cap U$ of~$G_1[V_a]$ that is coloured~$a$
(after uncolouring) is distributed as the independent set~$\bI$ generated by the procedure
described in Subsection~\ref{sub:generating independent set} applied to the graph~$G_1[V_a]$
with parameter
$\gamma'=\Delta/k=\Delta/\lceil\Delta/\gamma\rceil$.
Therefore, applying Theorem~\ref{thm:independent} with $\iota'=\iota/2$ first gives that
\begin{equation}\label{eq:Ia density}
\left|\Prob[v \in \bI_a] - \frac{1-e^{-\gamma'}}{\Delta}\right|\leq\frac{2}{\Delta^2},
\end{equation}
and second that
\begin{equation}\label{eq:probability Ia empty}
\frac{\Prob[\bI_a\cap N_{G_1}(v)\neq\varnothing]}{\E[|\bI_a\cap N_{G_1}(v)|]}< 1-\eps_{\ref{col_result}}(\sigma)+\frac{\iota}{2}
\end{equation}
for every $v\in V_1$ and $a\in L(v)$.

Set $\mu'=1-(1-e^{-\gamma'})/\gamma'$.
Recall that $\gamma'=\Delta/{\lceil\Delta/\gamma\rceil}=\gamma+O(1/\Delta)$, so
$\mu'=\mu+O(1/\Delta)$.
Since the set $V_1\setminus U$  of coloured vertices is the disjoint union
$\bigcup_{a}\bI_a$ on all the colours of~$L$, Equation~\eqref{eq:Ia density} gives
\[
\left|\Prob[v \in U] - \mu'\right| =
\left|\Prob[v \notin U] - \frac{1-e^{-\gamma'}}{\gamma'}\right| \leq
\sum_{a\in L(v)}\left|\Prob[v \in I_a] -\frac{1-e^{-\gamma'}}{\Delta} \right|
\leq |L(v)|\cdot\frac{2}{\Delta^2}
\leq \frac{2}{\Delta}
\]
for every~$v\in V_1$.
As a consequence, 
\begin{align}\label{eqn:vinU}
\left|\Prob[v\in U]-\mu\right| = O(1/\Delta).
\end{align}

Given~$v\in U$, let us estimate the expected size of $\del(v)$.
A fixed colour $a\in L(v)$ is in $\del(v)$
if at least one neighbour of~$v$ in $G$ is selected in $\bI_a$,
that is, if $N_G(v)\cap\bI_a\neq\varnothing$.
Since $N_{G}(v)$ is a subset of $N_{G_1}(v)$,
it follows from~\eqref{eq:probability Ia empty} that
$\Prob[a\in\del(v)]<
(1-\eps_{\ref{col_result}}(\sigma)+\iota/2)\E[|\bI_a\cap N_{G_1}(v)|]$.
Note further that
$\sum_{a\in L(v)}\E[|\bI_a\cap N_{G_1}(v)|]=\E[\bI_a\setminus U]\leq(1-\mu)\Delta+O(1)$.
Consequently,
\[
\E[|\del(v)|] = \sum_{a\in L(v)}\Prob[a\in\del(v)] \leq
\left(1-\eps_{\ref{col_result}}(\sigma)+\frac{\iota}{2}\right)(1-\mu)\Delta + O(1) = k-k'-\sqrt{\Delta}\log^2\Delta +O(1),
\]
and so we need to prove the following concentration statement.
\begin{claim}\label{claim:emptiness concentration}
  If $\Delta$ is large enough then \[
  \Prob[\left| |\del(v)| - \E[|\del(v)|]\right|
  \ge \sqrt{\Delta}\log\Delta]\le\Delta^{-\frac{1}{2}\log\log\Delta}.
    \]
\end{claim}

Moreover,
$\E[|N_G(u)\cap N_G(v)\cap U|]=\sum_{w\in N_G(u)\cap N_G(v)}\Prob[w\in U]=\mu|N_G(u)\cap N_G(v)|+O(1)$ by~\eqref{eqn:vinU}.
We use the following statement to exclude~\eqref{eq:bad2}.
\begin{claim}\label{claim:set concentration}
  Let $S$ be a set of vertices with $|S|\le \Delta$.
  Then if $\Delta$ is large enough,
    \[
    \Prob[~\left|\,|S\cap U|-\E[|S\cap U|]\,\right| \ge
      \sqrt{\Delta}\log^2\Delta~]\le\Delta^{-\frac{1}{2}\log\log\Delta}.
    \]
\end{claim}
Assuming Claims~\ref{claim:emptiness concentration} and~\ref{claim:set concentration},
we apply the Lov\'asz local lemma to show that, with positive probability, neither of the events~\eqref{eq:bad1} and~\eqref{eq:bad2} occurs.

The event~\eqref{eq:bad2} can only happen when~$u$ and~$v$ are at distance at most~$2$, so we restrict our attention to these cases.
The final state of a vertex in the procedure only depends on the state of the colours and priorities of its neighbours, so the final state of two vertices at distance at least~$4$ are independent.
Taking into account that~\eqref{eq:bad2} concerns the neighbourhood of~$u$ and~$v$,
the number of bad events correlated with~\eqref{eq:bad1} or~\eqref{eq:bad2}
is generously bounded by $d=\Delta^4(\Delta^2+1)<\Delta^7$.
The upper bound $p:=\Delta^{-\frac{1}{2}\log\log\Delta}$ on the probability of a bad event therefore satisfies
$4pd<1$ for $\Delta$ greater than some $\Delta_0$,
so the Lov\'asz local lemma proves the theorem.

It only remains to show that the random variables $|\del(v)|$ and $|N(u)\cap N(v)\cap U|$
are concentrated, as asserted in Claims~\ref{claim:emptiness concentration}
and~\ref{claim:set concentration}.
To do so, we use Theorem~\ref{thm:upward certs}.

\begin{proof}[Proof of Claim~\ref{claim:emptiness concentration}]
  Consider $X=|\del(u)|$ as a random variable on the product space
  $\Omega=\Pi_{v\in V}\Omega_v$,
  where the elements of $\Omega_v$ are the pairs $(c_0(v), \pi(v))$.
  We prove the concentration of $X$ by applying
  Theorem~\ref{thm:upward certs} with no exceptional outcome,
  that is, with $\Omega^*=\varnothing$.

  Let $\omega \in \Omega$. The coordinates that certify~$X(\omega)$ is large are the vertices of~$N(u)$,
  and for each $v\in N(u)$ that is uncoloured in~$\omega$,
  we choose one neighbour $\phi(v)$ of~$v$ that uncolours~$v$,
  that is, such that $c_0(v)=c_0(\phi(v))$ and $\pi(v)\ge \pi(\phi(v))$.
  We then define $I=N(u)\cup\phi(N(u)\cap U)$.
  The set~$I$ has size at most $2\Delta$. 

  Consider another outcome~$\omega'\in\Omega$.
  A colour~$a$ that is in $\del(u)$ for $\omega$ is also in $\del(u)$ for $\omega'$
  unless they differ on a coordinate $x$ corresponding to one of the two following
  situations:
  $c_0(x)=a$ for $\omega$ and $x$ is a neighbour of~$v$; or
  $c_0(x)=a$ for $\omega'$ and $x=\phi(v)$ for some neighbour~$v\in N(u)$ with colour $a$
  (for $\omega$).
  Note that such a vertex $x$ is part of~$I$ and can correspond to each situation for
  only one colour.
  As a consequence, if $\omega$ and $\omega'$ differ in fewer than $t/2$ coordinates of $I$,
  then $X(\omega')$ is at least $X(\omega)-t$.

  Applying Theorem~\ref{thm:upward certs} with $s=2\Delta$, $c=2$ and $t=\sqrt{\Delta}\log\Delta$ gives
\[
\Prob[|X-\E[X]|\ge \sqrt{\Delta}\log\Delta]
\le 4e^{-\frac{\log^2\Delta}{64}}
\]
which is less than $\Delta^{-\frac{1}{2}\log\log\Delta}$
  when $\Delta$ is large enough.
\end{proof}

\begin{proof}[Proof of Claim~\ref{claim:set concentration}]
  Consider $X=S\cap U$ as a random variable on the product space $\Omega=\prod_{v\in V}\Omega_v$,
  where the elements of $\Omega_v$ are the pairs $(c_0(v), \pi(v))$.
  Let us show that $X=|S\cap U|$ has an upward $(s,c)$-certificate.

  We first define a set~$\Omega^*$ of exceptional outcomes as
  \[
  \Omega^*=\left\{ |\sst{v\in N(u)}{c(u)=i}| \ge \log\Delta \text{ for some } i \in \N\right\},
  \]
  i.e. the set of events in which a particular colour appears
  (with respect to~$c_0$) more than $\log\Delta$ times in $S$.

  Let us estimate $\Prob[\Omega^*]$.
  For each colour~$x$, let~$s_x$ be the number of vertices of~$S$
  with $x$ in their list.
  Now, the probability that~$S$ contains more than $\log\Delta$ vertices coloured
  with~$x$ is at most 
  \[
  \sum_{i = \log\Delta}^{s_x}{\binom{s_x}{i}}
  \frac{\gamma^i}{\Delta^i}
  \le \sum_{i = \log\Delta}^{s_x}{\binom{\Delta}{i}}
  \frac{\gamma^i}{\Delta^i}
  \le \sum_{i = \log\Delta}^{s_x} \left(\frac{e\Delta}{i}\right)^i\frac{\gamma^i}{\Delta^i}
  \le s_x \cdot \left(\frac{\gamma e}{\log\Delta}\right)^{\log\Delta}.
  \]
  So by the union bound and a simple double-counting argument,
  \[
  \Prob[\Omega^*]\le\sum_{x\in\N}s_x\cdot \left(\frac{\gamma e}{\log\Delta}\right)^{\log\Delta}
  \le |S|\cdot k\cdot\left(\frac{\gamma e}{\log\Delta}\right)^{\log\Delta}
  \le \Delta^2\cdot\left(\frac{\gamma e}{\log\Delta}\right)^{\log\Delta}.
  \]
  For $\Delta$ large enough,
  $\Prob[\Omega^*] \le \Delta^{-\frac{2}{3}\log\log\Delta}$.

  Given an unexceptional outcome $\omega\in \Omega\backslash \Omega^*$,
  define a set of the coordinates that can certify 
  $X(\omega)$ is large as a set $I$ containing the vertices of $S\cap U(\omega)$,
  as well as for each $v\in S$ that is uncoloured in $\omega$,
  one neighbour~$\phi(v)$ that uncoloured $v$,
  that is, such that $c_0(\phi(v))=c_0(v)$ and $\pi(\phi(v))\leq\pi(v)$.
  This yields a set $I:=(S\cap U)\cup\phi(S\cap U)$ of size at most $2\Delta$.  

  Now, consider another unexceptional outcome $\omega'\in\Omega$ such that
  $X(\omega')\leq X(\omega)-t$ for some $t$.
  Note that a vertex $v\in S\cap U(\omega)$ remains uncoloured if the colours and
  priorities of~$v$ and~$\phi(v)$ remain the same,
  so $v$ remains uncoloured in $\omega'$
  unless $\omega'$ and $\omega$ differ in one of the coordinates $v$ or $\phi(v)$.
  Moreover, for a vertex $u\in I$, the set $\phi^{-1}(\{u\})$ is a monochromatic subset of~$S$
  and therefore has size at most $\log\Delta$ in an unexceptional outcome,
  so the coordinate $u$ is part of the certificate of at most $c:=\log\Delta+1$ vertices.
  As consequence, $\omega$ differs from $\omega'$ in at least $t/c$ coordinates.
  
  Applying Theorem~\ref{thm:upward certs} with $s=2\Delta$, $c=\log\Delta+1$
  and $t=\sqrt{\Delta}\log^2\Delta$ gives the bound
  \[
  \Prob[\left|X-\E[X]\right|\ge t] \le
  4e^{-\frac{\log^4\Delta}{32(\log\Delta+1)^2}}+4\Delta^{-\frac{2}{3}\log\log\Delta},
  \]
  which is less than $\Delta^{-\frac{1}{2}\log\log\Delta}$
  when $\Delta$ is large enough.
  This completes the proof of the claim.
\end{proof}
This completes the proof of the lemma.
\end{proof}

\subsection{Proof of Theorem~\ref{col_result}}\label{sub:proof}

Let us wrap things up.
We need to combine Lemma~\ref{lem:list colouring} with a result implicit in the proof of~\cite[Lem.~3.20]{BPP22}. 

\begin{lemma}[\cite{BPP22}]\label{lem:maintain-sparsity}
  For each $\iota>0$ and $0<\sigma<1$, there exists
  $\Delta_{\ref{lem:maintain-sparsity}}=\Delta_{\ref{lem:maintain-sparsity}}(\iota)$ 
  such that if $G$ is a  $\sigma$-sparse graph with maximum degree $\Delta\ge\Delta_{\ref{lem:maintain-sparsity}}$ then every $\mu$-quasirandom subgraph of $G$ is $(\sigma-\iota)$-sparse.
  \end{lemma}
 \begin{proof}
 We take the following fact, without proof, from the end of the proof of~\cite[Lem.~3.20]{BPP22}. For all vertices $u$ in a $\mu$-quasirandom subgraph $G'$ of $G$ 
 \[
 |E(G'[N_{G'}(u)])|<(1-\sigma){\binom{\Delta}2}+O(\Delta^{3/2}\log^5\Delta).
 \]
 Thus $G'$ is $(\sigma-\iota)$-sparse, where $\iota=O(\Delta^{-1/2}\log^5\Delta)$.  
 \end{proof}

To prove our main result, we repeatedly apply Lemma~\ref{lem:list colouring}.
We start with a comparatively small segment of the palette $[\lceil (1-\eps_{\ref{col_result}}(\sigma)+\iota)\Delta \rceil]$ as the list for every vertex.
In each iteration, we apply Lemma~\ref{lem:list colouring} to the residual subgraph. This yields a new residual subgraph of slightly smaller maximum degree but also of slightly smaller guaranteed list-sizes.
We must pad each of the lists using another small, so far unused, segment of $[\lceil (1-\eps_{\ref{col_result}}(\sigma)+\iota)\Delta \rceil]$ so that the ratio $\gamma$ in our application of Lemma~\ref{lem:list colouring} is consistent.
Note that Lemma~\ref{lem:maintain-sparsity} will help to ensure that we can still apply Lemma~\ref{lem:list colouring} in the following iteration, if so needed.
Once the residual subgraph has had its maximum degree fall below a certain threshold, we can greedily colour it to successfully complete the procedure.

In the proof, it will be handy to define for any graph $G$ the (local) sparsity $\sigma(G)$ as the largest $\sigma>0$ for which $G$ is $\sigma$-sparse.

\begin{proof}[Proof of Theorem~\ref{col_result}]
Fix $\iota$ and $\sigma$, and let $G$ and $\Delta=\Delta(G)$ be as in the statement, where $\Delta_{\ref{col_result}}=\Delta_{\ref{col_result}}(\iota)$ will satisfy certain inequalities to be specified during the proof.
We may assume $\iota\le1/3$ for otherwise the statement is trivial (as $\eps_{\ref{col_result}}(\sigma)\le1/3$ for $0<\sigma\le1$).
Let $\gamma>\max\{2,\gamma_{\ref{lem:list colouring}}\}$, where $\gamma_{\ref{lem:list colouring}}=\gamma_{\ref{lem:list colouring}}(\iota/3)$ is as given in Lemma~\ref{lem:list colouring}, and let $\mu = 1-(1-e^{-\gamma})/\gamma$. 
Let $k=\lceil\Delta/\gamma\rceil$, and let $L$ be the $k$-list-assignment defined by $L(u)=[k]$ for all $u\in V(G)$.
We iteratively colour the graph using Lemma~\ref{lem:list colouring}. 

Initialising $\Gamma=G$, $\Lambda=L$, $\kappa=k$, and $K=k$,  we perform the following procedure.
\begin{enumerate}
    \item\label{colstep1} Let $c$ be the partial proper $\Lambda$-colouring of $\Gamma$ given by Lemma~\ref{lem:list colouring}, with specific parameter choices $\gamma\leftarrow\gamma$, $\iota\leftarrow\iota/3$, $\sigma\leftarrow\sigma(\Gamma)$, $\Delta\leftarrow\Delta(\Gamma)$, $k\leftarrow \kappa$, and let $\Gamma_c$ be the $\mu$-quasirandom residual subgraph of maximum degree $\Delta'$ with residual list-assignment $\Lambda_c$ satisfying $|\Lambda_c(u)| \ge k'$ for each $u\in V(\Gamma_c)$.
    \item\label{colstep2}  Arbitrarily delete colours from each list so that each list in $\Lambda_c$ has size exactly $k'$. As $k'\le k'':=\lceil\Delta'/\gamma\rceil$ (see below) we add the elements $\{K+1,\dots,K+1+k''-k'\}$ to each list in order to maintain $\gamma$, the ratio of maximum degree to list-size.
    \item\label{colstep3} If $\Delta'\ge\lfloor\iota\Delta/3\rfloor$ replace $\Gamma$ with $\Gamma_c$, $\Lambda$ with $\Lambda_c$, $\kappa$ with $k''$, $K$ with $K+1+k''-k'$, and return to Step~\ref{colstep1}.
\end{enumerate}
In Step~\ref{colstep2} of the procedure, we used the fact that $k''=\lceil\Delta'/\gamma\rceil \ge k'$ which can be seen by comparing $k''$ to the value of $k'$ given in Lemma~\ref{lem:list colouring}:
\begin{align*}
\frac{k-k''}{\Delta(\Gamma)-\Delta'}
&= \frac{k-\left\lceil\Delta'/\gamma\right\rceil}{\Delta(\Gamma)-\Delta'}
=\frac{\gamma \left\lceil\Delta(\Gamma)/\gamma\right\rceil-\gamma \left\lceil\Delta'/\gamma\right\rceil}{\gamma(\Delta(\Gamma)-\Delta')} \\
&\le\frac1\gamma+\frac1{\Delta(\Gamma)-\Delta'}
<1-\eps_{\ref{col_result}}(\sigma(\Gamma))-\iota/3 +\frac1{\Delta(\Gamma)-\Delta'}
\leq \frac{k-k'+1}{\Delta(\Gamma)-\Delta'}
\end{align*}
where the penultimate inequality holds since $\gamma >2$, $\iota/3\le 1/9$, and $\eps_{\ref{col_result}}(\sigma)\le1/3$ for $0<\sigma\le1$.

In order to apply Lemma~\ref{lem:list colouring} in Step~\ref{colstep1} we must have that $\Delta(\Gamma)\ge\Delta_{\ref{lem:list colouring}}$ where $\Delta_{\ref{lem:list colouring}}=\Delta_{\ref{lem:list colouring}}(\iota/3,\gamma)$ is as given by Lemma~\ref{lem:list colouring} and so we want $\lfloor\iota\Delta_{\ref{col_result}}/3\rfloor\ge \Delta_{\ref{lem:list colouring}}$. 
With each application of Lemma~\ref{lem:list colouring} in Step~\ref{colstep1} we have by $\mu$-quasirandomness that
\[
\Delta(\Gamma_c)\le \Delta(\Gamma)\left(
\mu + \xi(\Delta(\Gamma))
\right),
\] 
where $\xi(x):=(\log^5 x)/\sqrt{x}$.
Note that $\xi(x)$ is decreasing in $x$ for all $x\ge e^{10}$ (say).
Therefore, since we always maintain that $\Delta(\Gamma)\ge \lfloor \iota\Delta/3\rfloor$, the maximum degree in the residual subgraph will decrease each time by at least a factor $\mu+\iota'$, where $\iota':=\xi(\lfloor \iota\Delta_{\ref{col_result}}/3\rfloor)$, provided $\lfloor \iota\Delta_{\ref{col_result}}/3\rfloor \ge e^{10}$.
It then follows that the procedure will terminate after at most 
$n=\lceil\log_{\mu + \iota'}(\iota/3)\rceil$ steps. 

Throughout the procedure, there may be some decrease in $\sigma(\Gamma)$, which we must control.
In particular, it will suffice to ensure that $\eps_{\ref{col_result}}(\sigma(\Gamma))\ge\eps_{\ref{col_result}}(\sigma)-\iota/3$ holds each time we apply Lemma~\ref{lem:list colouring}.
Note that $\eps_{\ref{col_result}}(x)$ is a continuous, strictly increasing function of $x$ for $0<x<1$.
Let $0<\sigma'<\sigma$ be the unique choice satisfying $\eps_{\ref{col_result}}(\sigma')=\eps_{\ref{col_result}}(\sigma)-\iota/3$ 
and let $\iota''=(\sigma-\sigma')/n$.
By Lemma~\ref{lem:maintain-sparsity} there is some $\Delta_{\ref{lem:maintain-sparsity}}=\Delta_{\ref{lem:maintain-sparsity}}(\iota'')$ such that if $\Delta(\Gamma)\ge \Delta_{\ref{lem:maintain-sparsity}}$ then $\sigma(\Gamma_c)\ge \sigma(\Gamma)-\iota''$. Thus throughout the procedure we have $\sigma(\Gamma)\ge \sigma-n\iota''=\sigma'$, implying that $\eps_{\ref{col_result}}(\sigma(\Gamma))\ge\eps_{\ref{col_result}}(\sigma')=\eps_{\ref{col_result}}(\sigma)-\iota/3$, as desired.
We therefore also want $\lfloor \iota\Delta_{\ref{col_result}}/3\rfloor \ge \Delta_{\ref{lem:maintain-sparsity}}$.

In summary, with $\Delta_{\ref{col_result}}=\frac{4}{\iota}\max\left\{\Delta_{\ref{lem:list colouring}},e^{10},\Delta_{\ref{lem:maintain-sparsity}}\right\}$, we run the procedure above (over at most $n$ iterations).
In each application of Lemma~\ref{lem:list colouring} we have sparsity $\sigma(\Gamma)$ close enough to $\sigma$ that the ratio of colours used, $k-k'$, to reduction in maximum degree, $\Delta(\Gamma)-\Delta'$, is at most
$1-\eps_{\ref{col_result}}(\sigma)+2\iota/3$. Since overall we reduce the maximum degree by at most $\Delta$, we ultimately obtain a partial proper colouring of $G$ using at most 
$(1-\eps_{\ref{col_result}}(\sigma)+2\iota/3)\Delta$ colours. 
Moreover, the ultimate residual subgraph has maximum degree less than $\lfloor\iota\Delta/3\rfloor$, and so we may complete to a proper colouring of $G$ greedily using at most $\iota\Delta/3$ additional colours. 
\end{proof}

\begin{rem}\label{rem:list}
The method of colouring that we employ requires that at each step we draw from a common reservoir of thus far unused colours. This is an obstacle to extending our results to list colouring. Note that the methods of Bonamy {\em et al.}~did include list colouring and the more general notion of correspondence colouring.
\end{rem}

\section{Two applications and one adaptation of Theorem~\ref{col_result}}\label{sec:applications}

\subsection{Proof of Theorem~\ref{strong_bound}}\label{sub:ErdosNesetrilproof}

As mentioned earlier, Bruhn and Joos~\cite{BrJo18} established a good upper bound on the local density in $L(G)^2$ for any $G$. Specifically, they showed that any edge $e$ of a graph $G$ has a neighbourhood in $L(G)^2$ that induces a subgraph of $L(G)^2$ with at most $1.5\Delta(G)^4+5\Delta(G)^3$ edges. 
One might hope for an even better result in this direction, as the corresponding local edge count in $L(G)^2$ for the hypothetical extremal graphs $G$ for the Erd\H{o}s--Ne\v{s}et\v{r}il conjecture is strictly less than $0.8\Delta(G)^4$.
However, another construction based on Hadamard codes shows that the bound is asymptotically best possible. In order to circumvent this obstacle Bonamy {\em et al.}~\cite{BPP22} showed it beneficial to restrict attention to a subgraph of high minimum degree --- as the other vertices may be coloured afterwards greedily --- and the considered subgraph has lower local density. More precisely, with the aid of Bruhn and Joos's result, they showed the following result, which we use here too combined with Theorem~\ref{col_result}.

\begin{theorem}[Bonamy {\em et al.}~\cite{BPP22}]\label{thm:BPP22}
Fix $0\le \eps\le 0.3$. For any graph $G$, let $H=L(G)^2$.
Let $F$ be a maximum subset of the vertices of $H$ that induces a subgraph of minimum degree $\lceil(2-\eps)\Delta(G)^2\rceil$. For any $f\in F$, the number of edges in the subgraph $H[N_{H[F]}(f)]$ induced by the neighbourhood of $f$ (in $H[F]$) is at most
\[
\left(\frac{31}{6}-\frac{128}{3(10-3\eps)}+4\eps-\eps^2\right)\Delta(G)^4.
\]
\end{theorem}

\begin{proof}[Proof of Theorem~\ref{strong_bound}]
Let $0<\iota,\sigma<1$ be constants to be chosen later in the proof, let $\Delta_0=\Delta_{\ref{col_result}}/2$, and let $G$ be a graph with $\Delta(G)\ge \Delta_0$.
Let $\eps=0.228$, $H=L(G)^2$, and $F\subseteq V(H)$ the subset as in Theorem~\ref{thm:BPP22}.
It suffices to show that the subgraph $H[F]$ induced by $F$ has chromatic number satisfying $\chi(H[F])\le \lceil(2-\eps)\Delta(G)^2\rceil$. For then if $\chi(H) > \lceil(2-\eps)\Delta(G)^2\rceil$, there would be some minimal $F' \supsetneq F$ with $\chi(H[F'])> \lceil(2-\eps)\Delta(G)^2\rceil$, and since $H[F']$ would then have minimum degree at least $\lceil(2-\eps)\Delta(G)^2\rceil$, this would contradict the choice of $F$.
By Theorem~\ref{thm:BPP22}, for any $f\in F$, the number of edges in the subgraph $H[N_{H[F]}(f)]$ induced by the neighbourhood of $f$ (in $H[F]$) is at most
\[
\left(\frac{31}{6}-\frac{128}{3(10-3\eps)}+4\eps-\eps^2\right)\Delta(G)^4
\le (1-\sigma)\binom{2\Delta(G)^2}{2},
\]
where we take the choice $\sigma=0.277$.
By Theorem~\ref{col_result}, $\chi(H[F])\le (1-\eps_{\ref{col_result}}(\sigma)+\iota)2\Delta(G)^2$, which one can easily verify is at most $(2-\eps)\Delta(G)^2$ if $\iota\le 0.0004$. This completes the proof.
\end{proof}

\subsection{Proof of Theorem~\ref{Reed_bound}}\label{sub:Reedproof}

In addition to Theorems~\ref{col_result} and~\ref{thm:Ree98}, we will use the following claimed result. Recall that a graph is said to be {\em $c$-critical} if it is not properly $c$-colourable, but every proper subgraph is.


\begin{theorem}[Delcourt and Postle~\cite{DePo17}]\label{thm:DePo17}
For each $\eps,\alpha>0$, if a graph $G$ is 
$\lceil(1-\eps)(\Delta(G)+1)\rceil$-critical
and $\omega(G) \le(1-\alpha)(\Delta(G)+1)$, then it is $(1-\alpha/2-\eps)(\alpha-2\eps)$-sparse.
\end{theorem}

\begin{proof}[Proof of Theorem~\ref{Reed_bound}]
Let $\iota>0$ be chosen small enough,
let $\Delta_0 = \max\{\Delta_{\ref{col_result}}(\iota),\Delta_{\ref{thm:Ree98}}\}$, and let $G$ be a graph with $\Delta(G)\ge \Delta_0$.
Writing $\zeta=0.119$, 
we wish to show that $\chi(G) \le \lceil(1-\zeta)(\Delta(G)+1)+\zeta\omega(G)\rceil$.
We may assume by Theorem~\ref{thm:Ree98} that $\omega(G) < (1-\eps_{\ref{thm:Ree98}})\Delta(G)$. 
Let $\alpha$ be $1-\omega(G)/(\Delta(G)+1)$ and note that $0<\alpha\le 1$.
For a contradiction, we may assume, without loss of generality, that $G$ is $\lceil(1-\zeta\alpha)(\Delta(G)+1)\rceil$-critical.
Now we may apply Theorem~\ref{thm:DePo17} with $\alpha$ and $\eps=\zeta\alpha$ to deduce that $G$ is $\sigma$-sparse with the choice
$\sigma= (1-\alpha/2-\zeta\alpha)(\alpha-2\zeta\alpha)$. 
We then have from Theorem~\ref{col_result} that $\chi(G) \le (1-\eps_{\ref{col_result}}(\sigma)+\iota)\Delta(G)$.
A quick computation with small enough $\iota>0$ checks that
\(
\zeta\alpha \le \eps_{\ref{col_result}}(\sigma)-\iota
\)
for $0<\alpha\le 1$, giving a contradiction to the fact that $G$ is $\lceil(1-\zeta\alpha)(\Delta(G)+1)\rceil$-critical. This completes the proof.
%
\end{proof}

\begin{rem}
To our knowledge, the work of~\cite{DePo17} has not yet completed peer review. A slightly weaker version of Theorem~\ref{thm:DePo17} follows from a result of Kelly and Postle~\cite[Thm.~4.1]{KePo20}. In particular, the graph $G$ can be guaranteed to be $(\alpha-\alpha^2/2-2\eps)$-sparse rather than $(1-\alpha/2-\eps)(\alpha-2\eps)$-sparse.
For the avoidance of any doubt, the reader might prefer us to apply that result instead, in which case we would deduce a weaker bound in the conclusion of Theorem~\ref{Reed_bound}, with $0.113$ and $0.887$ in place of $0.119$ and $0.881$.
\end{rem}

\subsection{Proof sketch for Theorem~\ref{Vu_bound}}\label{sub:Vuproof}

  In order to convince the reader of Theorem~\ref{Vu_bound},
  we explain how to prove Theorem~\ref{thm:independent} under the stronger assumption that the maximum codegree of $G$ is at most $(1-\cosigma)\Delta(G)$, where the value
  $\eps_{\ref{col_result}}(\sigma)$ is replaced by
  $\eps_{\ref{Vu_bound}}(\cosigma) = \max\{\cosigma/(1+2\cosigma) - (2\cosigma)^{3/2},\eps_{\ref{col_result}}(\cosigma)\}$ in the statement.
  Specifically, we show the following. 
  \begin{theorem}[Theorem~\ref{thm:independent} specialised to maximum codegree]\label{thm:independent with codegree}
    For every $\iota>0$, there are $\Delta_{\ref{thm:independent with codegree}}=\Delta_{\ref{thm:independent with codegree}}(\iota)$
    and $\gamma_{\ref{thm:independent with codegree}}=\gamma_{\ref{thm:independent with codegree}}(\iota)$ such that the following holds.
    Let $G$ be a $\Delta$-regular graph
    with maximum codegree at most $(1-\cosigma)\Delta$ and $\Delta\geq\Delta_{\ref{thm:independent with codegree}}$,
    and let
    $\bI$ be a random independent set obtained by the algorithm described
    in Section~\ref{sub:generating independent set} with some parameter
    $\gamma\geq\gamma_{\ref{thm:independent with codegree}}$.
    For every vertex $r\in V(G)$, 
    \[
    \left|\Prob[r\in \bI]-\frac{1-e^{-\gamma}}{\Delta}\right| \le \frac{2}{\Delta^2}.
    \]
    Moreover, setting $\bI_r = N(r) \cap \bI$, it holds that
    \[
    \frac{\Prob\left[\bI_r \neq \varnothing\right]}{\E\left[|\bI_r|\right]}
    \leq 
    1-\eps_{\ref{Vu_bound}}(\cosigma)+\iota.
    \]
  \end{theorem}

  \begin{proof}[Sketch of the proof]
    We explain the necessary modifications to the proof of Theorem~\ref{thm:independent}.

    We use the same setup and notation as in the proof of Theorem~\ref{thm:independent}.
    In particular, $\sigma$ denotes the local sparsity of $r$
    i.e.~$G[N(r)]$ contains exactly $(1-\sigma)\binom{\Delta}{2}$ edges.
    Using the intermediary claims of the proof of Theorem~\ref{thm:independent} and the extra assumption,
    we show the following alternative version of~\eqref{eq:pairs - triples},
    where $\eps_{\ref{col_result}}(\sigma)$ is substituted by
    $\eps_{\ref{Vu_bound}}(\cosigma)$:
    \begin{equation}\label{eq:pairs - triples for vu}
      \E[P_r-T_r] \geq \eps_{\ref{Vu_bound}}(\cosigma) + o(1).
    \end{equation}
    Note first that $\cosigma\leq\sigma - \sigma/\Delta=\sigma+o(1)$
    because~$G[N(r)]$
    has maximum degree at most~$(1-\cosigma)\Delta$
    and thus at most $(1-\cosigma)\Delta^2/2$ edges.
    Furthermore $\eps_{\ref{col_result}}(\cosigma)\le \eps_{\ref{col_result}}(\sigma)+o(1)$
    because $x\mapsto\eps_{\ref{col_result}}(x)$ is a smooth, increasing function.
    So in the case that
    $\eps_{\ref{col_result}}(\cosigma)$ is larger than the other argument
    $\cosigma/(1+2\cosigma) - (2\cosigma)^{3/2}$,
    \eqref{eq:pairs - triples for vu} is implied by its
    original form~\eqref{eq:pairs - triples}.
    We may therefore assume hereafter that
    $\eps_{\ref{Vu_bound}}(\cosigma)$ is equal to 
    $\cosigma/(1+2\cosigma) - (2\cosigma)^{3/2}$.
    This happens when $\cosigma$ is smaller than some constant smaller than $1/2$.
    Since $\eps_{\ref{col_result}}(2x)\geq\eps_{\ref{Vu_bound}}(x)$
    when $0<x<1/2$, we may also assume that $\sigma<2\cosigma$.
    
  Given a vertex~$u\in N(r)$, define $\sigma_u:= |N(r)\setminus N[u]|/\Delta$.
  In words, $\sigma_u\Delta$ is the number of nonneighbours of $u$ in $N(r)$.
  We know from the codegree hypothesis that $\sigma_u\geq\cosigma$.
  Let us set $\sigma_u'=\min\{\sigma_u, 1-\cosigma\}$.
  
  We first show how to estimate~$\E[P_r]$ given
  Claim~\ref{claim:pairs} in the proof of Theorem~\ref{thm:independent}.
  We claim that
  \begin{equation}\label{eq:pairs with d'}
  \frac{2}{2-\ell_{uv}}\geq \frac{2}{1+\sigma_u'+\sigma_v'}
  \end{equation}
  for every $uv\in\I_2$.
  Indeed, $u$ and $v$ have at least $\Delta-\sigma_u\Delta-\sigma_v\Delta$ common neighbours in~$N(r)$,
  so $\ell_{uv}\geq 1-\sigma_u-\sigma_v$.
  Further, this lower bound is negative whenever~$\sigma_u$ or~$\sigma_v$
  is larger than~$1-\cosigma$ because $\sigma_u,\sigma_v\geq\cosigma$,
  so capping these variables to $1-\cosigma$
  keeps the inequality valid.
  As a consequence, it also holds that $\ell_{uv}\geq 1-\sigma_u'-\sigma_v'$,
  which implies~\eqref{eq:pairs with d'}.

  We further bound $2/(2-\ell_{uv})$ from below by
  \[
  \frac{2}{1+\sigma_u'+\sigma_v'}=
  2-\frac{2\sigma_u'}{1+\sigma_u'+\sigma_v'}-\frac{2\sigma_v'}{1+\sigma_u'+\sigma_v'}\geq
  2-\frac{2\sigma_u'}{1+\sigma_u'+\cosigma}-\frac{2\sigma_v'}{1+\cosigma+\sigma_v'}=
  \frac{1-\sigma_u'+\cosigma}{1+\sigma_u'+\cosigma}+\frac{1+\cosigma-\sigma_v'}{1+\cosigma+\sigma_v'}.
  \]
  Summing on all $uv\in\I_2$,
  \begin{align*}
  \sum_{uv\in\I_2}\frac{2}{2-\ell_{uv}}&\geq
  \sum_{uv\in\I_2}\left(\frac{1-\sigma_u'+\cosigma}{1+\sigma_u'+\cosigma}+\frac{1+\cosigma-\sigma_v'}{1+\cosigma+\sigma_v'}\right)\\
  &= \sum_{u\in N(r)}\sigma_u\frac{1-\sigma_u'+\cosigma}{1+\sigma_u'+\cosigma}\\
  &\geq \sum_{u\in N(r)}\sigma_u'\frac{1-\sigma_u'+\sigma}{1+\sigma_u'+\cosigma}.
  \end{align*}
  Here the second line is deduced by double-counting,
  $\sigma_u$ being equal to the number of terms of the previous sum in which $u$ appears.
  Since the function $f:x\mapsto x\frac{1-x+\cosigma}{1+x+\cosigma}$ is concave on
  $[\cosigma,1-\cosigma]$,
  it is bounded from below by comparing evaluations on the endpoints of its range, namely
  $f(\cosigma)=\cosigma/(1+2\cosigma)$ and
  $f(1-\cosigma)=(1-\cosigma)\cosigma$.
  The latter is at least the former when $0\le \cosigma\leq1/2$,
  and so
  \[
  \sum_{uv\in\I_2}\frac{2}{2-\ell_{uv}}\geq\Delta\cdot\frac{\cosigma}{1+2\cosigma}.
  \]
  Now it follows from Claim~\ref{claim:pairs} in the proof of Theorem~\ref{thm:independent} that
  \begin{equation}\label{eq:pairs with codegree}
  \E[P_r]\ge
  \sigma\cdot\Eon{uv\in \I_2}\left[\frac{1}{2-\ell_{uv}}\right] + o(1)=
  \frac{1}{\Delta^2}\cdot\sum_{uv\in\I_2}\frac{2}{2-\ell_{uv}} + o(1) \geq
  \frac1\Delta\cdot\frac{\cosigma}{1+2\cosigma}+o(1).
  \end{equation}
  (Note that the first inequality follows from the fact that $x/((1-x)(2-x)\ge 0$ for $x\in[0,1)$.)



  We estimate $E[T_r]$ in rougher way.
  Following the proof of Claim~\ref{claim:pkeep triples} of Theorem~\ref{thm:independent}
  we have 
  \[\int_0^1\int_{z}^1\int_{y}^1p_{xyz}dxdydz\leq\frac{1}{\gamma^3},\]
  which further leads to $\Pk(u,v,w)\le6/\gamma^3$
  for every $uvw\in\I_3$.
  Consequently,
  \[
  \E[T_r]=\sum_{uvw\in\I_3}\Prob[u,v,w\in\A]\cdot\Pk(u,v,w)=
  \frac{\gamma^3}{\Delta^3}\cdot\sum_{uvw\in\I_3}\Pk(u,v,w)\leq\frac{6|\I_3|}{\Delta^3}.
  \]
  We know that $|\I_3|\leq\sigma^{3/2}\binom{\Delta}{3}$ by a result of Rivin~\cite{Riv02},
  so $\E[T_r]\leq \sigma^{3/2}$.
  
  Putting everything together gives
  \[
  \E[P_r-T_r] \geq \frac{\cosigma}{1+2\cosigma}-\sigma^{3/2}+o(1),
  \]
  which proves~\eqref{eq:pairs - triples for vu} because $\sigma<2\cosigma$.
\end{proof}

The remainder of the proof of Theorem~\ref{Vu_bound} follows the proof of Theorem~\ref{col_result}. More specifically, instead of Theorem~\ref{thm:independent} we use Theorem~\ref{thm:independent with codegree} to establish a codegree analogue of Lemma~\ref{lem:list colouring}. We then repeatedly apply that lemma as done with Lemma~\ref{lem:list colouring} in Subsection~\ref{sub:proof}. We omit the redundant details.

\section{Conclusion}

Our main contribution is a refinement of the na\"ive random colouring procedure through the analysis of a simple random priority assignment strategy. We have shown that this yields improved chromatic number bounds for graphs of bounded local density, for two distinct basic notions of local density, bounds that moreover are asymptotically optimal (in a precise sense) in the densest regime. This has yielded direct nontrivial progress in three longstanding conjectures due to Erd\H{o}s and Ne\v{s}et\v{r}il, Reed, and Vu. Although it is more than apparent that there is room for improvement in our bounds (and we hope and expect to encounter further progress along these lines in the coming years), the optimality in our bounds, while it is infinitesimal, marks an important milestone. Moreover our findings suggest that other random colouring procedures could benefit from a similar priority assignment strategy; see also~\cite{PeSr08}.

Let us remark that our approach directly yields randomised polynomial-time algorithms, specifically of time complexity that is linear in the number of vertices and polynomial in the maximum degree, for obtaining proper colourings using the claimed number of colours in Theorems~\ref{col_result} and~\ref{Vu_bound}. This follows by applying an algorithmic form of the Lov\'asz local lemma due to Moser and Tardos~\cite{MoTa10} and by simulating the regularisation process through the inclusion of dummy vertices.

One drawback of our methods is that they do not, to our knowledge, easily yield bounds on the list chromatic number of equivalent quantitative quality. Overcoming this obstacle (see Remark~\ref{rem:list}) would be of particular interest for Vu's conjecture, for example.

Theorems~\ref{col_result} and~\ref{Vu_bound} address the ``asymptotically dense'' regimes for two notions of bounded local density, the former for local average degree, the latter for local maximum degree. Our explorations naturally point towards studying the corresponding density regime for other notions of local density. For example, one could consider graphs of bounded (local) clique number, it then being arguably the most ``Ramsey-type'' colouring problem of this type. One referee kindly reminded us that Reed (in the paper from which Conjecture~\ref{conj:Ree98} originates) already showed the following.
\begin{theorem}[\cite{Ree98}]
\label{prob:Ramsey}
Define $\eps_{\ref{prob:Ramsey}} = \eps_{\ref{prob:Ramsey}}(\clsigma) = \clsigma/2$.
For each $0<\clsigma\le1/70000000$, there is $\Delta_{\ref{prob:Ramsey}}$ such that the chromatic number satisfies $\chi(G)\le (1-\eps_{\ref{prob:Ramsey}}(\clsigma))\Delta(G)+1/2$ for any graph $G$ with $\omega(G) \le (1-\clsigma)\Delta(G)$ and $\Delta(G)\ge\Delta_{\ref{prob:Ramsey}}$.
\end{theorem}
\noindent
Moreover, as $\clsigma\to0$, the factor $1/2$ in $\eps_{\ref{prob:Ramsey}}(\clsigma)$ is best possible, due to a probabilistic construction~\cite[Thm.~2]{Ree98}.
Another possibility perhaps worth pursuing is some analogue (of Theorems~\ref{col_result},~\ref{Vu_bound}, and~\ref{prob:Ramsey}) under the condition of bounded local Hall ratio.

%
%

%
%

\appendix 
\section*{Appendix}
\section{Exceptional outcomes and Talagrand's inequality}\label{sec:talagrand}


Theorem~\ref{thm:upward certs} was stated by Bruhn and Joos as is in terms of general probability spaces, but the result of Talagrand \cite{Tal95} from which the statement was derived was shown ignoring questions of measurability.  Talagrand elected to do this so as to make the proof more convenient and, as he observed~\cite[p.~11]{Tal95}, questions of measurability are largely irrelevant when proving such bounds since one can always use measurable approximations. With respect to algorithmic and combinatorial applications of Talagrand's inequality, the sample space is indeed usually discrete, in which case this discussion is moot. Nevertheless, one could argue that what was proved by Bruhn and Joos is as follows.

\begin{theorem}[Bruhn and Joos~\cite{BrJo18}]\label{thm:BrJo18}
The statement of Theorem~\ref{thm:upward certs} holds true under the additional assumption that $((\Omega_i,\sigma_i,\Prob_i))_{i=1}^n$ are discrete probability spaces.
\end{theorem}

\noindent
Our work makes important use of the sample space $[0,1]$, for which we have applied Theorem~\ref{thm:upward certs} rather than Theorem~\ref{thm:BrJo18}. 
For the elimination of any shadow of a doubt, we next explain how Theorem~\ref{thm:upward certs},
if true for finite probability spaces, also holds without any assumption on the probability spaces. That is, we show how Theorem~\ref{thm:BrJo18} implies Theorem~\ref{thm:upward certs}.
In order to do so, we approximate the random variables by random variables on finite spaces
by applying the following property.
\begin{proposition}\label{prop:finite image discretization}
  Let $(\Omega_i,\sigma_i,\Prob_i)_{i=1}^n$ be probability spaces
  and $(\Omega,\sigma,\Prob)$ be their product.
  Let $X:\Omega\to S$ be a random variable with values in a finite set~$S$.
  Fix $\epsilon>0$.
  For each $i\in[n]$, there is a finite subset $F_i\subseteq\Omega_i$ and
  a measurable function $\phi_i$ from $(\Omega_i,\sigma_i)$ to $(F_i, 2^{F_i})$,
  such that the function
  $\phi(\omega_1,\dots,\omega_n):=(\phi_1(\omega_1),\dots,\phi_n(\omega_n))$
  satisfies
  \[
  \Prob_{\omega\in\Omega}[X(\omega)\neq X(\phi(\omega))]<\epsilon.
  \]
\end{proposition}
Before proving this property, let us first use it together with Theorem~\ref{thm:BrJo18} to derive Theorem~\ref{thm:upward certs}.

\begin{proof}[Proof of Theorem~\ref{thm:upward certs}]
  We first prove the theorem under the extra assumption that~$X$ has a finite image
  and that~$\Prob[\Omega^*]<M^{-2}$.

  In order to approximate both~$X$ and~$\Omega^*$ with
  Proposition~\ref{prop:finite image discretization}, we define
  a random variable~$X^*$ such that
  $X^*(\omega)=(X(\omega), 0)$ for $\omega\notin\Omega^*$,
  and $X^*(\omega)=(X(\omega), 1)$ for $\omega\in\Omega^*$.
  The function $X^*:\Omega\to F\times\{0,1\}$ is measurable because
  the function~$X$ and the set~$\Omega^*$ are measurable.
  
  Let $\phi:\Omega\to\prod_{i=1}^nF_i\subseteq\Omega$ be as
  in the statement of Proposition~\ref{prop:finite image discretization} applied with the random variable $X^*$ and
  with some small enough $\epsilon>0$.
  The probability space $F_i$ is endowed with the discrete probability
  defined by $\Prob_i'[A]:=\Prob_i[\phi_i^{-1}(A)]$
  for every $A\subset F_i$.
  Let $(F=\prod_{i=1}^nF_i, 2^F, \Prob')$ be their product space and
  write~$X_{|F}$ the restriction of~$X$ to~$F$.

  Considering~$X_{|F}$ as a random variable on the finite product space~$F$,
  we aim to apply Theorem~\ref{thm:BrJo18} to $X_{|F}$.
  Let us define a set of exceptional outcomes as $\Omega^*_F=\Omega^*\cap F$.
  The measure of~$\Omega_F^*$, is estimated by
  \[
  \Prob'[\Omega_F^*]\leq \Prob[\Omega^*] + \epsilon,
  \]
  because 
  \[
  \phi^{-1}(\Omega_F^*)\subseteq\sst{\omega\in\Omega}{\omega\in\Omega^*\text{ or }X^*(\omega)\neq X^*(\alpha)}.
  \]
  In particular $\Prob'[\Omega_F^*]<M^{-2}$ if~$\epsilon$ is small enough.
  Moreover, the random variable $X_{|F}$ clearly has $(s,c)$-upward certificates
  because $F\setminus\Omega_F^*$ is a subset of $\Omega\setminus\Omega^*$.

  Theorem~\ref{thm:BrJo18} applied to $X_{|F}$ therefore gives
  \[
  \Prob[|X_{|F}-\E[X_{|F}]|\geq t]\leq4e^{-\frac{t^2}{16c^2s}}+4\Prob[\Omega^*_1]\leq
  4e^{-\frac{t^2}{16c^2s}}+4\Prob[\Omega^*]+4\epsilon.
  \]
  Recall that $\Prob_{\omega\in\Omega}[X^*(\omega)\neq X^*(\phi(\omega))]\leq\epsilon$,
  so $\Prob[X(\omega)\neq X(\phi(\omega))]\leq\epsilon$.
  Since moreover $|X(\phi(\omega)) - X(\omega)|\leq 2\sup |X|\leq2M$,
  the expectations $\E[X]$ and $\E[X_{|F}]$ differ by at most $2M\epsilon$ and
  \[
  \Prob_\Omega[|X-\E[X]|\geq t + 2M\epsilon]\leq
  \Prob_F[|X_{|F}-\E[X_{|F}]|\geq t] + \epsilon \leq
  4e^{-\frac{t^2}{16c^2s}}+4\Prob[\Omega^*]+5\epsilon.
  \]
  Letting $\epsilon$ tend to $0$ shows that
  $\Prob[|X-\E[X]> t]\leq 4e^{-\frac{t^2}{16c^2s}}+4\Prob[\Omega^*]$.
  The exact sought statement on $\Prob[|X-\E[X]\geq t]$ with the non-strict inequality
  automatically follows approximating $t$ from below.
  This concludes the proof in the case where $X$ has finite image and
  $\Prob[\Omega^*]<M^{-2}$.
  
  We now prove the result in the general case by approximating~$X$
  with a function with finite image.
  First, we can assume that $X$ is bounded.
  Indeed, the upward certificates of~$X$ easily implies that the values $X$ takes on
  $\Omega\setminus\Omega^*$ are at distance at most~$n/c$ from each other.
  Moreover, if~$X$ is unbounded, then $M=+\infty$ and $\Omega^*$ is a set of measure $0$,
  so the random variable~$X'$ that is equal to $X$ on $\Omega\setminus\Omega^*$ and
  to~$0$ on $\Omega^*$ is bounded and equal to $X$ almost everywhere.
  In this case, it is enough to show the result for~$X'$.

  Fix some $0<\epsilon<c$ and
  let~$X'$ be a rounding of $X$ defined by
  $X'=(1-\epsilon)\lfloor X/\epsilon^2\rfloor\epsilon^2$.
  The function $X'$ has $(s,c)$-upward certificates
  because if $X(\omega)< X(\omega')+t$ for some $t\geq c$, then
  $X'(\omega) - X'(\omega') \leq (1-\epsilon)(X(\omega) - X(\omega') + \epsilon^2) \leq (1-\epsilon)(t+\epsilon^2)$, which is smaller than $t$ provided that $\epsilon<c$.
  Moreover, $\sup|X'|\leq (1-\epsilon)(\sup|X|+\epsilon^2)<\sup|X|$ for $\epsilon$ small enough.
  In that case, $M':=\max(\sup|X|,1)$ is strictly smaller than~$M$
  unless $M=M'=1$. Since the result is trivial if $\Prob[\Omega^*]\geq\frac14$,
  we have $\Prob[\Omega^*]<M^{-2}$ in any case.

  The case of Theorem~\ref{thm:upward certs} proved above therefore applies to $X'$ and gives
  $\Prob[|X'-E[X']|\geq t]\leq 4e^{-\frac{t^2}{16c^2s}}+4\Prob[\Omega^*]$
  for $t>50c\sqrt{s}$.
  Since $|X'-X|\leq\epsilon^2$ on $\Omega$,
  $\Prob[|X-E[X]|\geq t+2\epsilon^2]\leq
  \Prob[|X'-E[X']|\geq t]\leq 4e^{-\frac{t^2}{16c^2s}}+4\Prob[\Omega^*]$.
  Again, letting $\epsilon$ tend to $0$ and approximating $t$ from below gives the sought bound.
\end{proof}

\subsection{Proof of Proposition~\ref{prop:finite image discretization}}

If~$A$ and~$B$ are two sets, $A\triangle B$ denotes the symmetric difference of~$A$ and~$B$,
that is~$A\triangle B=(A\setminus B)\cup(B\setminus A)$.
The proof of Proposition~\ref{prop:finite image discretization} is based on the following
property of measure theory about the structure of measurable sets.
\begin{proposition}[\cite{GHJSV15}, Lemma A.1]\label{prop:approximation of measurable}
  If~$(\Omega,\sigma,\Prob)$ is a probability space and~$\mathcal{S}$ is a nonempty family
  of measurable subsets of~$\Omega$ such that 
  \begin{itemize}
  \item $\mathcal{S}$ generates the $\sigma$-algebra $\sigma$; and
  \item $\mathcal{S}$ is stable under finite unions and complementary operations,
  \end{itemize}
  then for every~$\epsilon>0$ and every measurable set~$A$
  there is~$B\in\mathcal{S}$ such that $\Prob(A\triangle B)\leq\epsilon$.
\end{proposition}
Applied to product probability spaces,
Proposition~\ref{prop:approximation of measurable} has the following consequence.
\begin{proposition}\label{prop:boxes}
  Let $(\Omega_i,\sigma_i,\Prob_i)_{i=1}^n$ be probability spaces
  with product $(\Omega,\sigma,\Prob)$.
  Let $\mathcal{A}\subseteq \sigma_i$ be a finite set of measurable sets.
  For every $\epsilon>0$,
  there is a finite measurable partition $\Omega_i=\bigcup_{j=1}^{p_i}\Omega_{i,j}$
  for each $i\in[n]$ such that the following holds.
  Let $\B$ be the set of boxes of the form $\prod_{i=1}^n\Omega_{i,j(i)}$,
  then for every $A\in\mathcal{A}$ there is a set of
  boxes $\B_A\subseteq\B$ such that
  \[
  \Prob\left[A \triangle \bigcup_{B\in\B_A}B\right]\leq\epsilon.
  \]
\end{proposition}
\begin{proof}
  Consider the set $\mathcal{S}$ of union of boxes~$T$ of the form
  \[
  T=\bigcup_{i=1}^p\prod_{j=1}^nT_{i,j},
  \]
  where $p\in\N$ and
  $T_{i,j}$ ranges over the set~$\sigma_j$ of measurable sets of $\Omega_j$ for each
  $i\in[p]$ and $j\in[n]$.
  The set $\mathcal{S}$ satisfies the hypothesis of
  Proposition~\ref{prop:approximation of measurable},
  i.e. $\mathcal{S}$ is stable under finite unions and complementation and
  $\mathcal{S}$ generates the product $\sigma$-algebra $\sigma$.
  As a consequence, Proposition~\ref{prop:approximation of measurable}
  yields for each $A\in\mathcal{A}$ a union of boxes $T_A\in\mathcal{S}$
  such that $\Prob(A \triangle T_A)\leq\epsilon$.
  
  Now, it suffices to take for each $i\in [n]$,
  the partition $\Omega_i=\bigcup_{j=1}^{k_i}\Omega_{i,j}$
  generated by the measurable sets $T_{i,j}$ used in the box description of
  the $B\in\mathcal{B}_A$'s, so that
  each $T_A$ can be expressed as a union of boxes
  $T_A=\bigcup_{B\in\B_A}B$ for some subset $\B_A$ of
  the set of boxes of the form $\prod_{i=1}^n\Omega_{i,j(i)}$.
\end{proof}
It remains to prove Proposition~\ref{prop:finite image discretization}.
\begin{proof}[Proof of Proposition~\ref{prop:finite image discretization}]
  Applying Proposition~\ref{prop:boxes} to
  the finite set~$\mathcal{A}=\sst{X^{-1}(s)}{s\in S}$
  and error~$\frac{\epsilon}{2|S|}$
  gives
  partitions $\Omega_i=\bigcup_{j=1}^{k_i}\Omega_{i,j}$ and sets~$\B$
  and $(\B_A)_{A\in\mathcal{A}}$ as in the statement.
  By merging the elements $\Omega_{i,j}$ with probability~$0$ with
  an element $\Omega_{i,j'}$ with positive probability, we may further
  assume that $\Prob_i[\Omega_{i,j}]>0$ for every~$i$ and $j$.

  As a consequence, we can consider a random element~$x_{i,j}$ in $\Omega_{i,j}$
  (with respect to the conditional probability measure
  $\Prob_{i,j}[A]:=\Prob[A|\Omega_{i,j}]$), and set $\phi_i(x)=x_{i,j}$
  for every $x\in\Omega_{i,j}$.
  This gives rise to a random function $\phi(x_1,\dots,x_n):=(\phi_1(x_1),\dots,\phi_n(x_n))$
  from $\Omega$ to a finite subset of itself.

  Set $\Omega^*=\bigcup_{A\in\mathcal{A}} A \triangle (\bigcup_{B\in B_A}B)$ and
  consider the random variable
  \[
  E:=\sum_{\substack{B\in\B\\\phi(B)\subseteq\Omega^*}}\Prob[B].
  \]
  Since $\phi(B)$ is the singleton $\{(x_{1,j(a)},\dots,x_{n,j(n)})\}$
  if $B=\Prob_{i=1}^n\Omega_{i,j(i)}$, the condition $\phi(B)\subseteq\Omega^*$
  has to be read as $(x_{1,j(a)},\dots,x_{n,j(n)})\in\Omega^*$.
  The probability of this latter event is $\Prob[\Omega^*|B]$,
  so the expectation of $E$ is
  \[
  \E[E]=\sum_{B\in\B}\Prob[\Omega^*|B]\cdot\Prob[B]=\Prob[\Omega^*].
  \]
  Since further
  $\Prob[\Omega^*]\leq|\mathcal{A}|\cdot\epsilon/(2|S|)=\epsilon/2$,
  there exists $\phi$ such that $E\leq \epsilon/2$.

  We fix such a $\phi$. It remains to show $\Prob_{\omega\in\Omega}[X(\omega)\neq X(\phi(\omega))]\leq\epsilon$
  for this particular $\phi$.
  More precisely, we show that $X(\omega)= X(\phi(\omega))$
  whenever none of $\phi(\omega)$ and $\omega$ is in $\Omega^*$.
  Indeed, let $B_\omega\in\B$ be the box containing $\omega$ and $\phi(\omega)$.
  Setting $A=\sst{\omega'\in\Omega}{X(\omega')=X(\omega)}$,
  we have $\omega\in \bigcup_{B\in\B_A}B$
  because $\omega\in A$ and $\omega\notin A\triangle\bigcup_{B\in\B_A}B$,
  so $B_\omega\in\B_A$.
  Further, $\phi(\omega)\in B_\omega\in\B_A$
  so $\phi(\omega)\in A$
  because $\phi(\omega)\notin A\triangle\bigcup_{B\in\B_A}B$.
  We conclude that $X(\phi(\omega))=X(\omega)$ whenever none of $\phi(\omega)$ and $\omega$ is in $\Omega^*$, and thus 
\[\Prob_{\omega\in\Omega}[X(\omega)\neq X(\phi(\omega))]\leq \Prob[\omega \in \Omega^*]+\Prob[\phi(\omega)\in \Omega^*] \leq \Prob[\Omega^*]+\frac\epsilon2\leq\epsilon.\qedhere\]
\end{proof}

\section*{Acknowledgments} 
We are grateful to Luke Postle for bringing to our attention two important corrections upon an earlier version of this manuscript.
We are also thankful to several anonymous reviewers for their meticulous reading and for helpful comments and suggestions.

\bibliographystyle{amsplain}


\begin{aicauthors}
\begin{authorinfo}[eoin]
  Eoin Hurley\\
  Heidelberg University\\
  Heidelberg, Germany\\
  hurley\imageat{}informatik\imagedot{}uni-heidelberg\imagedot{}de\\
  \url{https://www.ifi.uni-heidelberg.de/theoi/team/eoin_hurley.html}
\end{authorinfo}
\begin{authorinfo}[remi]
  R\'emi de Joannis de Verclos\\
  Nijmegen, Netherlands\\
  remi\imagedot{}de\imagedot{}joannis\imagedot{}de\imagedot{}verclos\imageat{}ens-lyon\imagedot{}org\\
  \url{https://remi-de-verclos.github.io/}
\end{authorinfo}
\begin{authorinfo}[ross]
  Ross J. Kang\\
  University of Amsterdam\\
  Amsterdam, Netherlands\\
  ross\imagedot{}kang\imageat{}gmail\imagedot{}com\\
  \url{https://staff.fnwi.uva.nl/j.r.kang/}
\end{authorinfo}
\end{aicauthors}

\end{document}